\documentclass[12pt,reqno]{amsart}

\AtBeginDocument{\addtocontents{toc}{\protect\setlength{\parskip}{0pt}}}

\usepackage{amssymb}
\usepackage{amsrefs}
\usepackage{comment}
\usepackage{enumitem}
\usepackage{graphicx}
\usepackage{esint}
\usepackage{pdfpages}
\usepackage{svg}
\usepackage{caption}

\SetLabelAlign{parright}{\parbox[t]{\labelwidth}{\raggedleft#1}}

\numberwithin{equation}{section}

\PassOptionsToPackage{dvipsnames}{xcolor}

\newcommand{\Z}{\mathbb{Z}}
\newcommand{\R}{\mathbb{R}}
\newcommand{\C}{\mathbb{C}}
\newcommand{\mbS}{\mathbb{S}}
\newcommand{\Q}{\ms Q}
\newcommand{\B}{\ms B}
\newcommand{\mb}{\mathbf}
\newcommand{\mc}{\mathcal}
\newcommand{\ms}{\mathsf}

\newcommand{\delbar}{\bar \partial}
\newcommand{\del}{\partial}

\newcommand{\D}{\mathbb{D}}

\newcommand{\mbhat}[1]{\hat{\mb #1}}

\newcommand{\G}{G}

\usepackage[colorlinks=true,linkcolor=red,citecolor=blue]{hyperref}

\newcommand{\reg}{\mathrm{reg}}
\newcommand{\gr}{\mathrm{gr}}
\newcommand{\AD}{\mathrm{AD}}
\newcommand{\Lip}{\mathrm{Lip}}

\renewcommand{\H}{\mathbb{H}}
\renewcommand{\Im}{\operatorname{Im}}
\renewcommand{\Re}{\operatorname{Re}}
\usepackage{subcaption}

\DeclareMathOperator{\supp}{supp}

\newtheorem{theorem}{Theorem}[section]
\newtheorem{corollary}[theorem]{Corollary}

\newtheorem{proposition}[theorem]{Proposition}
\newtheorem{lemma}[theorem]{Lemma}

\theoremstyle{definition}
\newtheorem{definition}[theorem]{Definition}

\newtheorem*{example*}{Example}
\newtheorem*{claim*}{Claim}

\theoremstyle{remark}
\newtheorem*{remark*}{Remark}
\newtheorem*{remarks*}{Remarks}

\DeclareMathOperator{\Err}{Err}

\DeclareMathOperator{\SL}{SL}

\makeatletter
\def\namedlabel#1#2{\begingroup
    #2%
    \def\@currentlabel{#2}%
    \phantomsection\label{#1}\endgroup
}
\makeatother

\title{Fractal uncertainty in higher dimensions}
\author{Alex Cohen}
\thanks{Supported by an NSF Graduate Research Fellowship and a Hertz Foundation Fellowship. Partial support from NSF CAREER grant DMS-1749858 is acknowledged.}
\email{alexcoh@mit.edu}

\begin{document}

\begin{abstract}
We prove that if a fractal set in $\R^d$ avoids lines in a certain quantitative sense, which we call line porosity, then it has a fractal uncertainty principle. The main ingredient is a new higher dimensional Beurling--Malliavin multiplier theorem.
\end{abstract}

\maketitle

\section{Introduction}
\subsection{Main result}
A fractal uncertainty principle (FUP) says that a function cannot be localized to a fractal set in physical space and a fractal set in Fourier space at the same time. It has striking applications to quantum chaos---by applying FUP to fractal sets coming from chaotic dynamical systems, we can control high frequency waves on those systems. Bourgain and Dyatlov \cite{BourgainDyatlov} proved an FUP for sets in $\R$ satisfying a \textit{porosity} property, with applications to lower bounds for mass of eigenfunctions (Dyatlov, Jin,
and Nonnenmacher~\cites{DyatlovJinFullSupport,DyatlovJinNonnenmacher}), control for the Schr\"odinger equation and exponential decay for the damped wave equation~\cites{Jin-Control,Jin-DWE,DyatlovJinNonnenmacher}, and spectral gaps for open quantum systems (Dyatlov--Zahl and Dyatlov--Zworski \cites{DyatlovZahl,DyatlovZworski}).
See the surveys \cites{DyatlovSurvey,DyatlovJournees} for more details.

These results apply to surfaces because Bourgain and Dyatlov's FUP applies to subsets of $\R$. To show analogues for $d+1$ dimensional manifolds would need an FUP for subsets of $\R^d$. We prove such a result for any~$d\geq 1$, see below for definitions used ($h < 1/100$ denotes a small parameter):
\begin{theorem}\label{thm:FUP_higher_dim}
Let $\nu>0$ and assume that
\begin{itemize}
    \item $\mb X \subset [-1, 1]^d$ is $\nu$-porous on balls from scales $h$ to $1$, and 
    \item $\mb Y \subset [-h^{-1},h^{-1}]^d$ is $\nu$-porous on lines from scales $1$ to $h^{-1}$. 
\end{itemize}
Then there exist $\beta, C > 0$ depending only on $\nu$ and $d$ such that for all $f \in L^2(\R^d)$ 
\begin{equation}\label{eq:higher_dim_FUP_estimate}
    \supp \hat f \subset \mb Y\, \Longrightarrow\, \| f 1_{\mb X}\|_2 \leq C\, h^{\beta} \| f \|_2.
\end{equation}
\end{theorem}
One could remove the hypothesis that $\mb X \subset [-1,1]^d$ and $\mb Y \subset [-h^{-1},h^{-1}]^d$ using almost orthogonality in a similar way to~\cite{DyatlovJinNonnenmacher}*{Proposition 2.9}.
Porosity on lines is a stronger condition than porosity on balls and it is needed because of a counterexample in dimensions $d\geq 2$, see \eqref{eq:orthogonal_lines}. We believe that this is a natural assumption which can be established in applications.
We prove Theorem \ref{thm:FUP_higher_dim} by combining previous work of Han and Schlag \cite{HanSchlag} with a higher dimensional version of the Beurling--Malliavin multiplier theorem (Theorem \ref{thm:higher_dim_BM} below). This multiplier theorem is the main new ingredient, the proof involves an explicit construction of certain plurisubharmonic functions and H\"ormander's theorem on solvability of the $\delbar$ equation. The core of this paper is about constructing plurisubharmonic functions.

\subsection{Porosity and the one-dimensional case}

We say a set $\mb X \subset \R^d$ is \textit{$\nu$-porous on balls} from scales $\alpha_0$ to $\alpha_1$ if for every ball $\B$ of diameter $\alpha_0 < R < \alpha_1$ there is some $\mb x \in  \B$ such that $\B_{\nu R}(\mb x) \cap \mb X = \emptyset$. Here $\B_{\nu R}(\mb x)$ is the radius $\nu R$-ball about $\mb x$.
Similarly, we say a set $\mb X$ is \textit{$\nu$-porous on lines from scales $\alpha_0$ to $\alpha_1$} if for all line segments $\tau$ with length $\alpha_0 < R < \alpha_1$, there is some $\mb x \in \tau$ such that $\B_{\nu R}(\mb x) \cap \mb X = \emptyset$.\footnote{Directional porosity is an existing notion similar to line porosity, see Chousionis's paper \cite{Chousionis}.}
We always assume $\nu \leq 1/3$. For subsets of $\R$ porosity on balls is the same as porosity on lines, and we just say a set is \textit{porous}. 

\begin{figure}
\centering
\begin{minipage}[c]{.45\linewidth}
  \centering
  \includegraphics[height=2in]{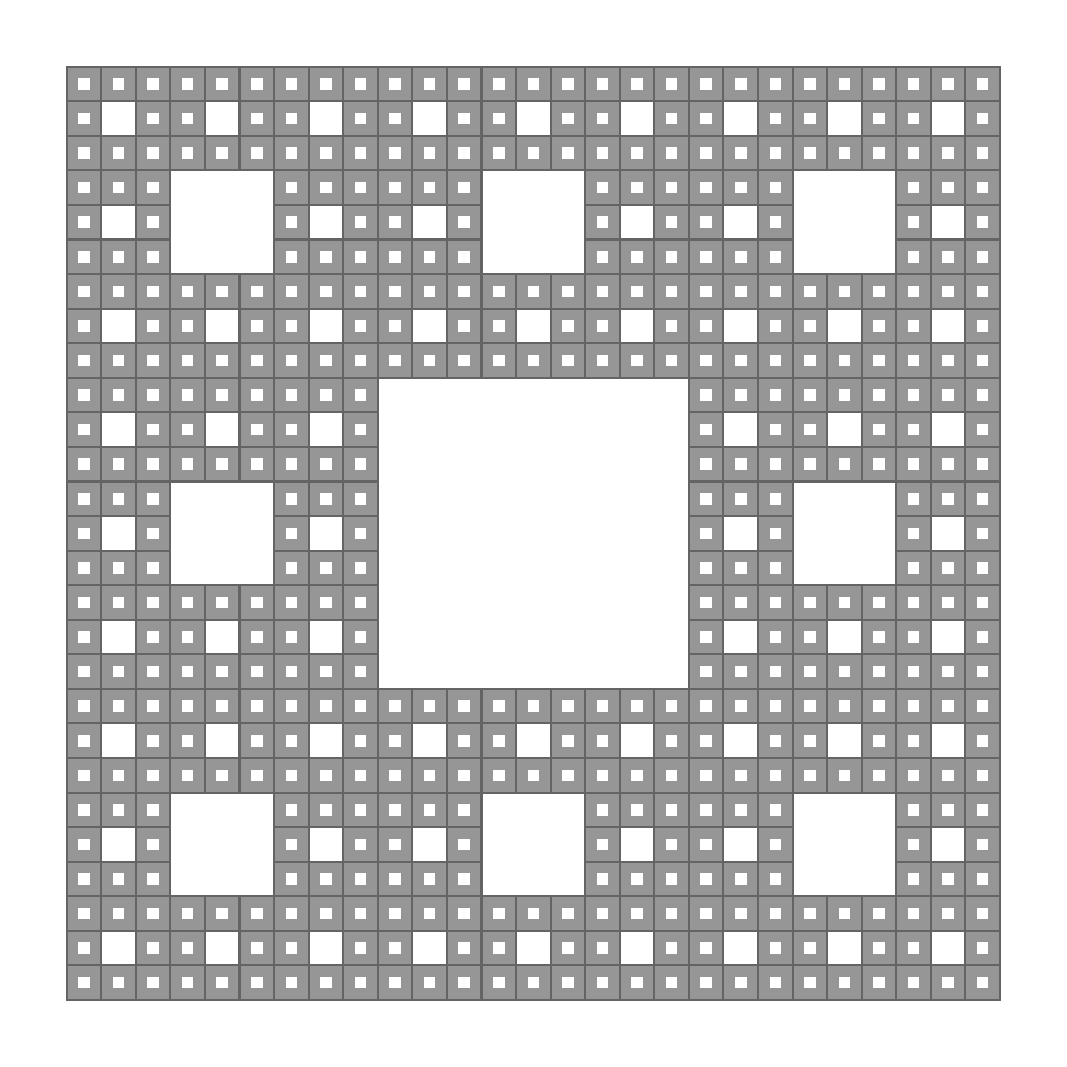}
  \captionsetup{width=\linewidth}
  \caption{The Sierpinski carpet is porous on balls but not porous on lines}
  \label{fig:sierpinski_carpet}
\end{minipage}
\hspace{0.3in}
\begin{minipage}[c]{.45\linewidth}
  \centering
  \includegraphics[height=2in]{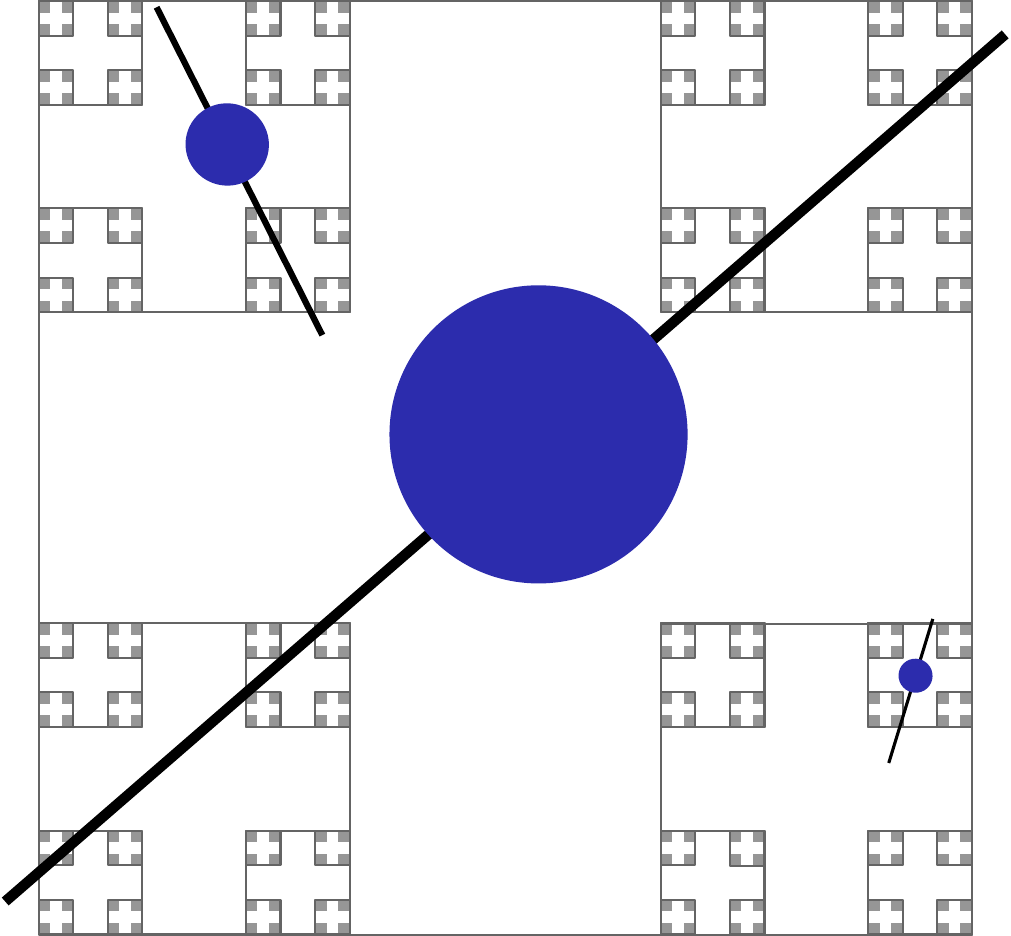}
  \captionsetup{width=\linewidth}
  \caption{The product of two middle thirds Cantor sets is porous on lines}
  \label{fig:prod_middle_thirds}
\end{minipage}
\end{figure}

Porosity on lines is the stronger condition. For example, any line is porous on balls but not porous on lines. See Figure \ref{fig:sierpinski_carpet} for another set which is porous on balls but not lines, and Figure \ref{fig:prod_middle_thirds} for a set which is porous on lines.
We now state Bourgain and Dyatlov's main theorem, which is the one dimensional case of Theorem \ref{thm:FUP_higher_dim}.
\begin{theorem}[Bourgain-Dyatlov]\label{thm:FUP_1D}
Let 
\begin{itemize}
    \item $\mb X \subset [-1,1]$ be $\nu$-porous from scales $h$ to $1$, and
    \item $\mb Y \subset [-h^{-1},h^{-1}]$ be $\nu$-porous from scales $1$ to $h^{-1}$.
\end{itemize}
There exists $\beta, C > 0$ depending only on $\nu$ such that for all $f \in L^2(\R)$
\begin{equation}
    \supp \hat f \subset \mb Y \, \Longrightarrow\, \| f 1_{\mb X} \|_{2} \leq Ch^{\beta}\, \| f\|_2.
\end{equation}
\end{theorem}

\begin{remark*}
In Bourgain and Dyatlov's paper the hypothesis is that $\mb X$ and $\mb Y$ are Ahlfors--David regular rather than porous. These two notions are equivalent up to a change in parameters: any regular set of dimension $<1$ is porous, and any porous set is contained in a regular set of dimension $<1$. The first statement of FUP using porous sets appeared in \cite{DyatlovJinFullSupport}. 
\end{remark*}

Not all porous sets $\mb X, \mb Y \subset \R^d$ have a fractal uncertainty principle. Speaking formally, we could have 
\begin{equation}\label{eq:orthogonal_lines}
\mb X = \{(t, 0)\, :\, t \in \R\},\quad \mb Y = \{(0, t)\, :\, t \in \R\}.
\end{equation}
If $\mu_{\mb X}$ and $\mu_{\mb Y}$ are the standard measures on $\mb X$ and $\mb Y$, then $\widehat {\mu_{\mb X}} = \mu_{\mb Y}$. See \cite{DyatlovSurvey}*{Example 6.1} for more details. We need porosity on lines to rule out this example.

\subsection{Prior work on higher dimensional fractal uncertainty}

A set $\mb X \subset \R^d$ is \textit{Ahlfors--David $\delta$-regular} with constant $C_{\AD}$ from scales $\alpha_0$ to $\alpha_1$ if there is a measure $\mu$ supported on $\mb X$ satisfying the following. For every ball $\B$ with diameter $\alpha_0 < R < \alpha_1$,
\begin{equation}
    \mu(\B) \leq C_{\AD}\, R^{\delta},
\end{equation}
and if in addition $\B$ is centered at a point in $\mb X$, then
\begin{equation}
    \mu(\B) \geq C_{\AD}^{-1}\, R^{\delta}.
\end{equation}
For $\mb X \subset [-1,1]^d$ a $\delta$-regular set from scales $h$ to $1$ and $\mb Y\subset [-h^{-1}, h^{-1}]^d$ a $\delta'$-regular set from scales $1$ to $h^{-1}$, there is a trivial bound 
\begin{equation}
\label{e:trivial-bound}
    \supp \hat f \subset \mb Y\, \Longrightarrow\, \| f 1_{\mb X}\|_2  \leq C\min(1, h^{(d-(\delta+\delta'))/2}) \| f \|_2
\end{equation}
where $C$ depends only on $\delta, \delta', C_{\AD}, d$. The estimate $\| f 1_{\mb X}\|_2  \leq Ch^{(d-(\delta+\delta'))/2} \| f \|_2$ follows from combining $L^1\to L^{\infty}$ boundedness of the Fourier transform with a volume bound on the sets $\mb X$ and $\mb Y$. 
An FUP is any improvement over this trivial bound, and the regimes $\delta+\delta' < d$ and $\delta+\delta' > d$ are quite different. Recently, Backus, Leng, and Z. Tao \cite{BackusLengTao} gave a definitive result in the former setting. They proved an FUP if $\delta+\delta' < d$ and $\mb X, \mb Y$ are not orthogonal in a certain sense. The present paper is about the $\delta + \delta' > d$ regime (the bound~\eqref{eq:higher_dim_FUP_estimate} trivially follows from~\eqref{e:trivial-bound} if $\delta+\delta' < d$). 

Han and Schlag \cite{HanSchlag} proved an FUP when $\mb X$ is an arbitrary porous set and $\mb Y$ is a Cartesian product of one dimensional porous sets. Cladek and T. Tao \cite{CladekTao} proved an additive energy estimate for fractal sets and used this to prove an FUP when the ambient dimension $d$ is odd and $\mb X, \mb Y$ are $\delta$-regular with $d/2 - \varepsilon(d, C_{\AD}) < \delta < d/2+\varepsilon(d, C_{\AD})$. The author \cite{Cohen} proved an FUP when $\mb X, \mb Y$ are Cantor sets in $\Z/N\Z \times \Z/N\Z$ which don't contain a pair of orthogonal lines (the ideas in the current paper are unrelated to that work). 

We also mention that Dyatlov~\cite{DyatlovExpositoryNote} wrote an expository note giving an alternative
point of view on some of the proofs in the present paper.

\subsection{The Beurling--Malliavin multiplier problem}

A key ingredient in Bourgain and Dyatlov's proof of Theorem \ref{thm:FUP_1D} is the Beurling--Malliavin (BM) multiplier theorem, a classical result in harmonic analysis. This theorem has been revisited many times by many authors, see in particular Beurling and Malliavin's original paper \cite{BM_original} and the recent survey by Mashregi, Nazarov, and Havin \cite{NazarovHavinMashregi}.

\begin{theorem}[Beurling--Malliavin]{\label{thm:BM}}
Let $\omega: \R \to \R_{\leq 0}$ be a weight function satisfying 
\begin{align}
    |\omega(x_1) - \omega(x_2)| &\leq C_{\Lip} |x_1 - x_2|\quad \text{for all $x_1, x_2 \in \R$},\label{eq:BM_cond_lip} \\ 
    \int_{\R} \frac{\omega(x)}{1+x^2}\, dx &> -\infty. \label{eq:BM_cond_growth} 
\end{align}
For every $\sigma > 0$, there is a nonzero function $f \in L^2(\R)$ such that $\supp \hat f \subset [-\sigma, \sigma]$ and
\begin{align*}
    |f(x)|& \leq e^{\omega(x)} \qquad \text{for all $x \in \R$}.
\end{align*}
\end{theorem}
Condition \eqref{eq:BM_cond_lip} asserts Lipschitz regularity and \eqref{eq:BM_cond_growth} controls the growth of $\omega$. 
\begin{remarks*} \leavevmode
\begin{enumerate}[label=\arabic*.]
    \item Here are some examples to help digest the growth condition \eqref{eq:BM_cond_growth}.
    \begin{enumerate}[label=(\roman*)]
        \item If $\omega(x) = -|x|$ then the growth condition is not satisfied.
        \item If $\omega(x) = -\frac{|x|}{\log(2+|x|)}$ then the growth condition is not satisfied.
        \item If $\omega(x) = -\frac{|x|}{(\log(2+|x|))^2}$ then the growth condition is satisfied.
    \end{enumerate}

    \item If we look for functions with Fourier support in $[0,\infty)$ rather than $[-\sigma, \sigma]$ there is a very precise result. 
    Given a measurable function $\omega: \R \to [-\infty,\infty)$, the following are equivalent:
    \begin{enumerate}[label=(\roman*)]
        \item There exists $f \in L^2(\R)$ with $\supp \hat f \subset [0,\infty)$ and $|f| = e^{\omega}$,\label{thmitem:supp_in_half_line}
        \item We have $e^{\omega} \in L^2(\R)$ and the growth condition \eqref{eq:BM_cond_growth} is satisfied.\label{thmitem:growth_satisfied}
    \end{enumerate}
    \vspace{0.1in}
    The direction \ref{thmitem:supp_in_half_line}$\Rightarrow$\ref{thmitem:growth_satisfied} is called the second F. \& M. Riesz Theorem, and the direction \ref{thmitem:growth_satisfied}$\Rightarrow$\ref{thmitem:supp_in_half_line} is the construction of outer functions. See \cite{NazarovHavinMashregi}*{\S1.1}. The direction \ref{thmitem:supp_in_half_line}$\Rightarrow$\ref{thmitem:growth_satisfied} shows that the growth condition in Theorem \ref{thm:BM} is necessary. 
\end{enumerate}
\end{remarks*}
We prove a higher dimensional Beurling--Malliavin theorem which is the key ingredient for Theorem \ref{thm:FUP_higher_dim}. We hope this result will be of independent interest. Let $\omega: \R^d \to \R_{\leq 0}$ be a weight function and define
\begin{align}
    \G(\mb x) &= \int_{1/2}^{2} |\omega(s\mb x)|\, ds, \quad  \label{eq:defn_of_H}\\ 
    \G^*(r) &= \sup_{|\mb x| = r} G(\mb x). \label{eq:defn_of_G*} 
\end{align}
Also let $\langle \mb x \rangle = (1+|\mb x|^2)^{1/2}$. 
\begin{theorem}\label{thm:higher_dim_BM}
Let $\omega: \R^d \to \R_{\leq 0}$ be a weight satisfying 
\begin{align}
    \omega(\mb x) &= 0 && \text{for $|\mb x| \leq 2$}, \label{eq:higher_dim_BM_weight_zero_small_x}\\
    |D^a\omega (\mb x)| &\leq C_{\reg} \langle \mb x \rangle^{1-a} && \text{for $0 \leq a \leq 3$}, \label{eq:regularity_condition_0to3}\\ 
    \int_0^{\infty} \frac{G^*(r)}{1+r^2}\, dr &\leq C_{\gr}. \label{eq:higher_dim_BM_growth_cond}
\end{align}
For any $\sigma > 0$, there exists a function $f \in L^2(\R^d)$ such that 
\begin{align}
    \supp \hat f &\subset \B_{\sigma}, \\ 
    |f(\mb x)| &\geq \frac{1}{2} && \text{for all $\mb x \in \B_{r_{\min}}$}, \\ 
    |f(\mb x)| &\leq Ce^{c\sigma\, \omega(\mb x)} && \text{for all $\mb x \in \R^d$}.
\end{align}
We may take
\begin{align}
    c &= \frac{c_d}{\max(C_{\reg}, C_{\gr})} \\
    r_{\min} &= c_d\, \min(\sigma, \sigma^{-1}) \\ 
    C &= C_d\, \max(\sigma^{-C_d}, e^{3\sigma})
\end{align}
where $c_d, C_d > 0$ are constants that depend only on the dimension.
\end{theorem}
The constants blow up as $\sigma \to 0$ because the condition $\supp \hat f \subset \B_{\sigma}$ becomes very hard to satisfy, and they blow up as $\sigma \to \infty$ because the condition $|f(\mb x)| \leq Ce^{c\sigma \omega(\mb x)}$ becomes very hard to satisfy. Only the constant $c$ depends on $C_{\reg}$ and $C_{\gr}$ whereas $r_{\min}$ and $C$ are given in terms of the ambient dimension $d$ and spectral radius $\sigma$. 

The regularity condition \eqref{eq:regularity_condition_0to3} is a Kohn-Nirenberg symbol condition up to three derivatives. 
Setting $a = 0$ gives the mild growth condition $|\omega(\mb x)| \leq C_{\reg}\langle \mb x \rangle$, and setting $a = 3$ gives the 3rd derivative condition $|D^3\omega(\mb x)| \leq C_{\reg}\langle \mb x\rangle^{-2}$. 
Theorem \ref{thm:higher_dim_BM} is much weaker than the Beurling--Malliavin theorem in one dimension because we require a lot more regularity. Nevertheless, the weights we construct for fractal sets will satisfy \eqref{eq:regularity_condition_0to3}. 

Let us discuss for a moment the growth condition \eqref{eq:higher_dim_BM_growth_cond}. 
On the one hand, taking $\omega \to G$ smooths out $\omega$ and makes it grow less quickly. On the other hand, $G \to G^*$ is a maximum and makes it grow more quickly. Morally, $G^*$ is constant on dyadic scales $[2^j, 2^{j+1}]$. Notice that in one dimension, 
\begin{align*}
    \int_0^{\infty} \frac{G^*(r)}{1+r^2}\, dr \sim \int_{-\infty}^{\infty}\frac{|\omega(t)|}{1+t^2}\, dt
\end{align*}
up to constants on both sides, so \eqref{eq:higher_dim_BM_growth_cond} is the same growth condition on $\R$ as in the classical Beurling--Malliavin theorem. The proof of Theorem \ref{thm:higher_dim_BM} involves estimating different dyadic pieces and then summing them together. We can get a decent estimate for each dyadic piece using only the regularity of $\omega$, and \eqref{eq:higher_dim_BM_growth_cond} is needed to sum these contributions.
The growth condition controls the mass of $\omega$ on lines through the origin, which makes sense in view of our observation that fractal sets can only have an FUP if they avoid lines.

\begin{remark*}
The condition \eqref{eq:higher_dim_BM_weight_zero_small_x} that $\omega(\mb x) = 0$ for $|\mb x| \leq 2$ is not really necessary in Theorem \ref{thm:higher_dim_BM}. 
One could modify a weight to satisfy \eqref{eq:higher_dim_BM_weight_zero_small_x} up to a change in constants. 
\end{remark*}

\subsection{Outline of the proof of fractal uncertainty}

We use a result of Han and Schlag \cite{HanSchlag} to deduce Theorem \ref{thm:FUP_higher_dim} from Theorem \ref{thm:higher_dim_BM}. The statements below are not exactly as they appear in Han and Schlag's paper, see \S\ref{sec:han_schlag_comparison} for a comparison. The factor $\sqrt{d}$ appears in this section in converting between $\ell_1$ and $\ell_2$ norms on $\R^d$.

\begin{definition}\label{defn:damping_func_l2}
The set $\mb Y \subset \R^d$ \textit{admits a damping function} with parameters $c_1, c_2, c_3, \alpha\in (0, 1)$ if there exists a function $\psi \in L^2(\R^d)$ satisfying
\begin{align}
    \supp \hat \psi &\subset \B_{c_1}, \\ 
    \| \psi \|_{L^2(\B_1)} &\geq c_2, \\ 
    | \psi(\mb x)| &\leq \langle \mb x \rangle^{-d} && \text{for all $\mb x \in \R^d$}, \\ 
    | \psi(\mb x)| &\leq \exp\left(-c_3 \frac{|\mb x|}{(\log (2+|\mb x|))^{\alpha}}\right)&& \text{for all $\mb x \in \mb Y$}.\label{eq:damping_func_decay_Y}
\end{align}
\end{definition}
It is important that $\alpha < 1$. If instead $\alpha > 1$ then \eqref{eq:damping_func_decay_Y} could hold on all of $\R^d$ and the definition wouldn't be interesting. Because $\alpha < 1$ the damping function has to decay much faster on $\mb Y$ than it does on the rest of $\R^d$.
Conditional on the existence of damping functions, Han and Schlag proved the following FUP. 

\begin{theorem}[\cite{HanSchlag}*{Theorem 5.1}]\label{thm:FUP_conditional_damping}
Suppose that
\begin{itemize}
    \item $\mb X \subset [-1,1]^d$ is $\nu$-porous on balls from scales $h$ to $1$, and
    \item $\mb Y \subset [-h^{-1}, h^{-1}]^d$ satisfies the following. There exist $c_2,c_3,\alpha \in (0,1)$ such that for all $h < s < 1$ and $\eta \in [-h^{-1}s-5, h^{-1}s+5]^d$ the set 
    \begin{equation}
        s\mb Y + [-4,4]^d + \eta
    \end{equation}
    admits a damping function with parameters $c_1 = \frac{\nu}{20\sqrt{d}}$, and $c_2, c_3, \alpha$. 
\end{itemize}
 Then there exists $\beta = \beta(\nu, c_2, c_3, d, \alpha) > 0$ and $\widetilde C = \widetilde C(\nu, c_2, c_3, d, \alpha) > 0$ so that for all $f \in L^2(\R^d)$
\begin{equation}\label{eq:estimate_han_schlag}
    \supp \hat f \subset \mb Y\, \Longrightarrow\, \| f 1_{\mb X}\|_2 \leq \widetilde C h^{\beta}\, \| f \|_{2}.
\end{equation}
\end{theorem}
The proof has three steps.
\begin{enumerate}[label=\arabic*.]
    \item Prove the following quantitative unique continuation principle for functions with Fourier support in $\mb Y$. Let $w > 0$ be a small parameter and let $\{\Q_{\mb n}\}_{\mb n \in \Z^d}$ be a collection of width-$w$ cubes, exactly one in each integer cube. Set
    \begin{equation*}
        U = \bigcup_{\mb n \in \Z^d} \Q_{\mb n},\quad \Q_{\mb n} \subset \mb n + [0,1]^d.
    \end{equation*}
    There is some $c > 0$ so that for any $f \in L^2(\R^d)$ and $U$ as above,
    \begin{equation*}
        \supp \hat f \subset \mb Y \, \Longrightarrow\, \| f 1_U\|_2 \geq c \| f\|_2.
    \end{equation*}
    The proof starts by convolving $f$ with the damping functions for $\mb Y$ to get functions with Fourier decay like $\hat f(\xi) \lesssim \exp\Bigl(-\frac{|\xi|}{(\log(2+|\xi|))^{\alpha}}\Bigr)$ for some $\alpha < 1$. The problem is then to prove unique continuation for functions with rapidly decaying Fourier transform.
    \item Use the quantitative unique continuation principle from the last step to obtain a single-scale estimate. For $h < r < 1$, roughly speaking
    \begin{equation*}
        \supp \hat f \subset \mb Y \, \Longrightarrow\, \| f1_{\mb X + \B_{r/L}} \|_2 \leq (1-c) \| f 1_{\mb X + \B_{r}}\|
    \end{equation*}
    where $L > 0$ is a large constant. In Han and Schlag's paper smooth cutoffs are used rather than indicator functions. This estimate means that at every scale $h < r < 1$, $f$ has some fixed portion of its mass in the holes of the porous set $\mb X$. 
    \item Iterate the single scale estimate $\sim \log h^{-1}$ many times to obtain the power saving bound \eqref{eq:estimate_han_schlag}.
\end{enumerate}
This is the same strategy that Bourgain and Dyatlov developed for Theorem \ref{thm:FUP_1D}. Jaye and Mitkovski \cite{JayeMitkovski} abstracted the unique continuation part of this argument.
The main contribution of \cite{HanSchlag} is proving the right unique continuation principle for functions on $\R^d$ that have rapidly decaying Fourier transform.

We use our Theorem \ref{thm:higher_dim_BM} to make damping functions for line porous sets, which combined with Theorem \ref{thm:FUP_conditional_damping} gives Theorem \ref{thm:FUP_higher_dim}, see \S\ref{sec:finishing_pf_main_theorem} for details.
\begin{proposition}\label{prop:existence_damping}
Let $\mb Y \subset [-3h^{-1}, 3h^{-1}]^d$ be $\nu$-porous on lines from scales $\mu > 1$ to $h^{-1}$. Then there is some $\alpha = \alpha(\nu) < 1$ such that for any $0 < \sigma < 1$, $\mb Y$ admits a damping function with parameters $\alpha$ and
\begin{align}
    c_1 &= \sigma, \\ 
    c_2 &= c(\mu, d)\, \sigma^{C_d},\\ 
    c_3 &= c(\nu, d)\, \sigma.
\end{align}
\end{proposition} 

\begin{remarks*}
\begin{enumerate}[leftmargin=*,label=\arabic*.]
    \item In practice we will take $\mu = 10\sqrt{d}/\nu$.
    \item The quantitative dependence on $\sigma$ is of the same form as \cite{JinZhang} (in $d = 1$, they have $c_2 \propto \sigma^6$).
    \item We will be able to take 
\begin{equation*}
    \alpha(\nu) = 1 - c\frac{\nu}{|\log \nu|}
\end{equation*}
for some absolute constant $c > 0$. 
\end{enumerate}
\end{remarks*}

\subsection{An application of fractal uncertainty}

We sketch Dyatlov and Jin's lower bound for eigenfunctions. Our goal is just to give a sense of how FUP is applied and we ignore many important details. This section is about one dimensional FUP, not higher dimensions.

Let $M$ be a connected compact hyperbolic surface. Write $\psi_k$ as the $L^2$-normalized $k$th Laplace eigenfunction with eigenvalue $\lambda_k = h^{-2}$. 
A fundamental question in quantum chaos is how the mass of high-frequency eigenfunctions is distributed. Dyatlov and Jin give information in this direction.

\begin{theorem}[Dyatlov \& Jin \cite{DyatlovJinFullSupport}]\label{thm:DyatlovJin}
Let $U \subset M$ be a nonempty open set. For some $c_U > 0$,
\begin{equation*}
    \| \psi_k 1_U \|_2 \geq c_U\quad \text{for all $k > 0$}.
\end{equation*}
\end{theorem}
Here is a rough sketch of how the fractal uncertainty principle is used to prove this Theorem. 
We can write $M = \Gamma \backslash \D$ where $\D$ is the Poincar\'e disk and $\Gamma\subset\SL(2,\mathbb R)$ is a group of isometries. Then $\psi_k$ lifts to a $\Gamma$-invariant eigenfunction $\widetilde \psi_k$ on $\D$, and $U$ lifts to a $\Gamma$-invariant open subset $\widetilde U \subset \D$. 

For $b\in \mathbb \mbS^1$ and $z\in\mathbb D$, denote by $P_b(z)$ the Poisson kernel.
For any $(b,r)\in \mathbb \mbS^1 \times \R$, the \emph{hyperbolic plane wave}
\begin{equation}
    \psi_b^r(z) := P_b(z)^{\frac12+ir},\quad z\in\mathbb D
\end{equation}
solves the eigenfunction equation
$-\Delta\psi_b^r=(r^2+\frac14)\psi_b^r$ on~$\mathbb D$. 
If $r > 0$ we call this an outgoing wave and if $r < 0$ it is incoming. Because $\psi_k$ has eigenvalue $\lambda_k$, we take $r = \sqrt{\lambda_k - 1/4}$. 

We can synthesize $\widetilde \psi_k$ in two ways, using either outgoing or incoming waves:
\begin{align*}
    \widetilde \psi_k(z) = \int_{\mathbb \mbS^1} f(b) \psi_b^r(z)\, db,\quad \widetilde \psi_k(z) = \int_{\mathbb \mbS^1} g(b) \psi_b^{-r}(z)\, db,\quad r \sim h^{-1},
\end{align*}
where $f,g$ are distributions on~$\mathbb \mbS^1$.
These distributions are related by an explicit formula 
(see e.g.~\cite{Borthwick}*{\S4.4})
\begin{equation}\label{eq:hyperbolic_fourier_transform}
    g(b) = c_r\int_{\mathbb \mbS^1} e^{-(1+2ir) \log |b-a|}\, f(a)\, da.
\end{equation}

Now let $\varepsilon > 0$ be small enough that $B_{1-\varepsilon } \subset \D$ covers $M$. Let $\gamma$ be a geodesic on $\D$ with endpoints $\gamma_+, \gamma_- \in \mathbb \mbS^1$. 
Define
\begin{equation}
    \mb X = \bigcup \{\gamma_+, \gamma_-\}\quad \text{over all $\gamma$ such that $\gamma \cap B_{1-\varepsilon } \neq \emptyset$ and $\gamma \cap \widetilde U = \emptyset$}.
\end{equation}
The set $\mb X \subset \mathbb \mbS^1$ represents the geodesics on $M$ that do not intersect $U$.
Using unique ergodicity of the horocycle flow on $M$ one can show that $\mb X$ is porous.

Morally speaking, if $\| \psi_k 1_U \|_2 = o(1)$, then $f$ and $g$ are both localized $h$-close to the set $\mb X$ where $h = \lambda_k^{-1/2}$. Because $f$ and $g$ are related by an oscillatory integral \eqref{eq:hyperbolic_fourier_transform}, the fractal uncertainty principle applied to the $h$-neighborhood of $\mb X$ rules out this scenario. 
See the survey \cite{DyatlovJournees} for details. 

It is conjectured that Theorem \ref{thm:DyatlovJin} holds in higher dimensions as well. Suppose $M$ is a $d$-dimensional hyperbolic manifold and $U \subset M$ is an open subset. We get fractals $\mb X \subset \mathbb \mbS^{d-1}$ in the same way, and if $\mb X$ is line porous then our Theorem \ref{thm:FUP_higher_dim} applies.

\subsection{Outline of the paper}
In \S\ref{sec:BM_problem_intro} we discuss how the Beurling--Malliavin multiplier problem naturally splits into two steps.
\begin{enumerate}[label=Step \arabic*:,leftmargin=*,labelindent=\parindent]
    \item Plurisubharmonic Beurling--Malliavin \eqref{BM_higher_dim_step:1} is a potential theory problem about constructing plurisubharmonic functions.
    \item Analytic Beurling--Malliavin \eqref{BM_higher_dim_step:2} is a several complex variables problem about constructing entire functions from those pluri\-subharmonic functions.
\end{enumerate}
Towards the end of \S\ref{sec:BM_problem_intro} we state our solution to each of these steps and give the proof of Theorem \ref{thm:higher_dim_BM} modulo the results of \S\S\ref{sec:constr_psh_Eomega_C}-\ref{sec:constr_entire_from_psh}.
In \S\ref{sec:constr_psh_Eomega_C} we define an extension operator taking functions on $\R^d$ to functions on $\C^d$ and use this operator to construct plurisubharmonic functions. 
In \S\ref{sec:modifying_weight_funcs} we show how to take a weight function satisfying the hypotheses of Theorem \ref{thm:higher_dim_BM} and modify it so the construction in \S\ref{sec:constr_psh_Eomega_C} is applicable. 
Together, \S\ref{sec:constr_psh_Eomega_C} and \S\ref{sec:modifying_weight_funcs} complete \ref{BM_higher_dim_step:1} and form the core of this paper. 
In \S\ref{sec:constr_entire_from_psh} we complete \ref{BM_higher_dim_step:2} using H\"ormander's $L^2$ theory of the $\delbar$ equation. This section follows an unpublished note of Bourgain. In \S\ref{sec:finishing_pf_main_theorem} we prove Proposition \ref{prop:existence_damping} and finish the proof of Theorem \ref{thm:FUP_higher_dim}.
In Appendix \ref{sec:loose_ends} we prove some loose ends.

\subsection{Notation}
For $f \in L^2(\R^d)$, we use the Fourier transform
\begin{align*}
    \hat f(\xi) &= \int_{\R^d} f(\mb x)\, e^{-2\pi i\, \mb x \cdot \xi}\, d\mb x.
\end{align*}
We often denote vectors $\mb z \in \C^d$ by $\mb z = \mb x + i \mb y$, with $\mb x, \mb y \in \R^d$. We use $\mbhat y$ to denote a unit vector, and if $\mb y \in \R^d \setminus \{0\}$ we write $\mbhat y = \mb y / |\mb y|$. The $\ell_2$ norm on $\R^d$, $\C^d$ is denoted $|\mb x|, |\mb z|$. We let
\begin{equation*}
    \langle \mb x\rangle = (1+|\mb x|^2)^{1/2}.
\end{equation*}
We denote the Hilbert transform on $L^2(\R)$ by $f \mapsto H[f]$. For functions $f \in C_0^1(\R)$, this is given by 
\begin{equation}\label{eq:defn_of_hilbert_transform}
    H[f](x) = p.v. \int_{-\infty}^{\infty} \frac{f(x-t)}{t}\, \frac{dt}{\pi}.
\end{equation}
For $u \in C^2(\C^d)$, $\del \delbar u$ is a Hermitian form which can be represented in coordinates as the Hermitian matrix 
\begin{align*}
    \langle (\del\delbar u) \mbhat e_j, \mbhat e_k\rangle &= \frac{\partial^2u }{\del z_j \delbar z_k} \\
    &= \frac{1}{4}(\partial_{x_j}\partial_{x_k} + \partial_{y_j}\partial_{y_k})u  + \frac{1}{4}i (\partial_{x_j} \partial_{y_k} -  \partial_{x_k} \partial_{y_j}) u
\end{align*}
where $\mbhat e_j = (0, \ldots, 0, 1, 0, \ldots, 0)$. 

For functions $f \in C^2(\R^d)$, the quadratic form $D^2f(\mb x)$ applied to the vector $\mb v$ is given by
\begin{equation}
    \langle (D^2 f(\mb x)) \mb v, \mb v \rangle. 
\end{equation}
We denote $D^a f = (\partial^{\alpha} f)_{|\alpha| = a}$ where $\alpha$ ranges over multi indices, and 
\begin{equation}
    |D^a f(\mb x)| = \sup_{|\alpha| = a} |\partial_{\alpha} f(\mb x)|.
\end{equation}
We use $A \lesssim B$ to denote that $A \leq C_d B$ where $C_d > 0$ only depends on the ambient dimension. We use $c_d, C_d > 0$ to denote small/large constants depending only on the dimension which may change from line to line. 

\subsection*{Acknowledgements}
Thanks to Tuomas Sahlsten for pointing out prior work on directional porosity. Thanks to Larry Guth and Ruixiang Zhang for several helpful discussions. Many thanks to Semyon Dyatlov for detailed and helpful comments, and for several useful conversations along the way. Thanks to anonymous referees for helpful comments that improved the paper. 

\section{The Beurling--Malliavin multiplier problem}\label{sec:BM_problem_intro}
The Beurling--Malliavin (BM) problem is about constructing functions with bounded Fourier support that have certain decay properties. In this section we discuss the one dimensional Beurling--Malliavin problem and then outline our approach in higher dimensions. In the outline we state the main results of \S\ref{sec:modifying_weight_funcs} and \S\ref{sec:constr_entire_from_psh} and use these to prove Theorem \ref{thm:higher_dim_BM}.

The starting point for the BM problem is the Paley--Wiener characterization of functions with bounded Fourier support. 

\begin{theorem}[Paley--Wiener]\label{thm:paley_wiener}
A function $f \in L^2(\R^d)$ has Fourier support in $\B_{\sigma/2\pi} = \Bigl\{\mb \xi\, :\, |\mb \xi| \leq \frac{\sigma}{2\pi}\Bigr\}$ if and only if $f$ is the restriction to $\R^d$ of an entire function $\tilde f: \C^d \to \C$ such that 
\begin{equation}
    |\tilde f(\mb x + i \mb y)| \leq A\, e^{\sigma |\mb y|} \quad \text{for some $A > 0$.}\label{eq:PaleyWiener_critetion}
\end{equation}
\end{theorem}
See \S\ref{subsec:PaleyWienerProof} for a proof sketch and \cite{HormanderVol1}*{Theorem~7.3.1} for a full proof. 

\subsection{Beurling--Malliavin in \texorpdfstring{$\R$}{R}}
Let us start in one dimension. We are given a weight function $\omega: \R \to \R_{\leq 0}$ with $\omega(0) = 0$. We would like to find a nonzero entire function $f: \C \to \C$ such that
\begin{align}
    \log |f(x)| &\leq \omega(x) \quad \text{for $x \in \R$}, \label{eq:BM_f_decay}\\
     \log |f(0)| &\geq -1, \label{eq:BM_f0=1}\\ 
    \log |f(x + iy)| &\leq \sigma|y| + A'.\label{eq:BM_PW}
\end{align}
Equation \eqref{eq:BM_f_decay} quantifies the decay of $f$, equation \eqref{eq:BM_f0=1} quantifies the non-vanishing of $f$, and equation \eqref{eq:BM_PW} ensures the Paley--Wiener criterion is satisfied so $\supp \hat f \subset [-\sigma/2\pi , \sigma/2\pi]$.

A function $u: \C \to \R$ is \textit{subharmonic} if it is upper semicontinuous and satisfies $\Delta u \geq 0$ in the distributional sense. If $f$ is an entire function then $\log |f|$ is subharmonic on $\C$. In fact, if $Z(f)$ is the zero locus of $f$ and $\mu_{Z(f)}$ is the counting measure on $Z(f)$ then
\begin{equation}
    \Delta \log |f| = 2\pi\mu_{Z(f)}. \label{eq:laplace_log_entire}
\end{equation}
Not much is lost by viewing $\log |f|$ as a general subharmonic function, and this is the best way to think about the magnitude of $f$. 

If we could solve the BM problem we could find a subharmonic function $u: \C \to \R$ such that $u \leq \omega$ on $\R$, $u(0) = 0$, and $u(x+iy) \leq \sigma|y|$. Several of the proofs work by finding a converse to this situation. There are two steps: the subharmonic Beurling--Malliavin problem and the analytic Beurling--Malliavin problem.

\begin{description}[style=multiline,leftmargin=\widthof{\bfseries $\mc{SH}$-BM.}+0.25in,align=parright]
  \item[\namedlabel{BM_1D_step:1}{$\mc{SH}$-BM}] Find a subharmonic function $u: \C \to \R$ such that $u|_{\R} \leq \omega$, $u(0) = 0$, and $u(x+iy) \leq \sigma|y|$.
  \item[\namedlabel{BM_1D_step:2}{$\mc A$-BM}] Find an analytic function $f$ such that $\log |f| \lesssim u$ and $f(0) = 1$. 
\end{description}
Each of these steps are approachable problems. First let's discuss the subharmonic BM problem. 
As a first attempt, one could try solving
\begin{description}[style=multiline,leftmargin=\widthof{\bfseries Exact $\mc{SH}$-BM.}+0.25in,labelindent=\parindent]
\item[\namedlabel{BM_1D_exact_step:1}{Exact $\mc{SH}$-BM}] Find a subharmonic function $u: \C \to \R$ such that $u|_{\R} = \omega$ and $u(x+iy) \leq \sigma |y|$. 
\end{description}
A natural candidate solution is to take $u = E\omega + C|y|$ where $E\omega: \C \to \R$ is obtained by separately harmonically extending $\omega$ to the upper and lower half planes.  We compute 
\begin{equation}
    \Delta u = 2\partial_y u(x+i0)\delta_{\R} = 2(H[-\omega'] + C)\delta_{\R}.\label{eq:laplace_harmonic_ext_1D}
\end{equation}
The operator $\omega \to H[-\omega']$ arises as the Dirichlet-to-Neumann operator of $\C \setminus \R$. As long as 
\begin{equation}\label{eq:hilbert_of_deriv_bdd}
 \| H[\omega'] \|_{\infty} < \infty
\end{equation}
we can take $C\geq\| H[\omega'] \|_{\infty}$ and $u$ will be subharmonic on $\C$ as desired. If \eqref{eq:hilbert_of_deriv_bdd} holds, then the \ref{BM_1D_exact_step:1} problem is solved in a canonical way.

The main challenge of Theorem \ref{thm:BM} is solving \ref{BM_1D_step:1} under the weaker condition that $\omega$ is just Lipschitz and Poisson integrable---in general the solution will have $u|_{\R} \leq \omega$ rather than $u|_{\R} = \omega$. 
There have been many approaches to this problem over the years. 
In their original paper Beurling \& Malliavin \cite{BM_original} use a variational argument based on the energy method for Dirichlet's problem. 
Koosis \cite{Koosis} developed an approach based on Perron's method of subsolutions for the Dirichlet problem. 
Mashreghi, Nazarov, and Havin \cite{NazarovHavinMashregi} solve \ref{BM_1D_step:1} by explicitly manipulating a Lipschitz weight $\omega$ to a modified weight $\widetilde \omega \leq \omega$ which satisfies $\| H[\widetilde \omega']\| < \infty$.
It turns out that the weights we care about for fractal uncertainty satisfy \eqref{eq:hilbert_of_deriv_bdd} so this main challenge is not relevant to us.

Now let's turn to the analytic BM problem. We start with a subharmonic function $u: \C \to \R$ and want to construct an entire function $f$ with $\log |f| \sim u$. Most proofs solve \ref{BM_1D_step:2} by carefully choosing the zero locus of $f$ in view of \eqref{eq:laplace_log_entire} and then writing $f$ as a Weierstrass product.
Bourgain wrote an unpublished note \cite{BourgainHormander} giving a different approach to \ref{BM_1D_step:2} based on H\"ormander's $L^2$ theory for the $\delbar$ equation. The upshot is that there exists an entire function $f: \C \to \C$ such that $f(0) = 1$ and 
\begin{equation*}
    \int_{\C} |f(z)|^2 e^{-u(z)} < C. 
\end{equation*}
This $L^2$ bound can be converted to an $L^{\infty}$ bound using subharmonicity, solving \ref{BM_1D_step:2}.

\subsection{Beurling--Malliavin in \texorpdfstring{$\R^d$}{Rd}}

In the last section we saw that the best way to think about the log of the magnitude of an entire function on $\C$ is as a general subharmonic function. How should we think about the magnitude of an entire function on $\C^d$?

Let $f: \C^d \to \C$ be entire. Then $\log |f|$ is a \textit{plurisubharmonic} function. A function $u: \C^d \to \R$ is plurisubharmonic if it is upper semicontinuous and its restriction to every complex line is subharmonic. Written explicitly, this means that we have the sub-mean value property
\begin{equation}\label{eq:sub_mean_value_prop_psh}
    u(\mb z) \leq \fint_0^{2\pi} u(\mb z + e^{i\theta}\mb v)\, d\theta \quad \text{ for any $\mb z,\mb v \in \C^d$}.
\end{equation}
Here $\fint_0^{2\pi} = \frac{1}{2\pi} \int_0^{2\pi}$ is a mean value. See the beginning of \S\ref{subsec:pf_prop_Ew_psh} for more discussion of the sub-mean value property and equivalence with the definition in terms of a nonnegative Laplacian.
A $C^2$ function $u$ is plurisubharmonic if the Hermitian matrix $\left(\frac{\del u}{\del z_j \delbar z_k}\right)$ is positive semidefinite. 
It is an important insight from several complex variables that the best way to think about the log of the magnitude of an entire function is as a general plurisubharmonic function. 
Once again we split the Beurling--Malliavin problem into two steps: the plurisubharmonic BM problem and the analytic BM problem.

\begin{description}[style=multiline,leftmargin=\widthof{\bfseries $\mc{PSH}$-BM.}+0.25in,align=parright]
  \item[\namedlabel{BM_higher_dim_step:1}{$\mc{PSH}$-BM}] Find a plurisubharmonic function $u: \C^d \to \R$ such $u|_{\R^d} \leq \omega$, $u(0) = 0$, and $u(\mb x+i\mb y) \leq \sigma|\mb y|$.
  \item[\namedlabel{BM_higher_dim_step:2}{$\mc{A}$-BM}] Find an entire function $f$ such that $\log |f| \lesssim u$ and $f(0) = 1$. 
\end{description}
This is the same two steps but with plurisubharmonic in place of subharmonic and vectors $\mb x, \mb y$ in place of scalars $x, y$. 

The main difficulty when we try to extend the Beurling--Malliavin theorem to higher dimensions is solving the plurisubharmonic BM problem.
As a first step, in \S\ref{sec:constr_psh_Eomega_C} we solve 
\begin{description}[style=multiline,leftmargin=\widthof{\bfseries Exact $\mc{PSH}$-BM.}+0.25in,labelindent=\parindent]
\item[\namedlabel{BM_higher_dim_exact_step:1}{Exact $\mc{PSH}$-BM}] Find a plurisubharmonic function $u: \C^d \to \R$ such that $u|_{\R^d} = \omega$ and $u(\mb x+i\mb y) \leq \sigma |\mb y|$. 
\end{description}
In equation \eqref{eq:defn_of_extension_op} we define an extension operator $\omega \to E\omega$  which takes a function on $\R^d$ to a function on $\C^d$. 
In one dimension $E$ is the Poisson extension operator, and in higher dimensions it is a suitable generalization. 
Proposition \ref{prop:Ew_is_psh} says that if $\omega$ satisfies two conditions (labeled \ref{item_Ew_psh:hilbert} and \ref{item_Ew_psh:second_deriv}) then $E\omega + C|\mb y|$ is plurisubharmonic on $\C^d$. 
Condition \ref{item_Ew_psh:hilbert} says that \eqref{eq:hilbert_of_deriv_bdd} holds uniformly for every restriction of $\omega$ to a line. Condition \ref{item_Ew_psh:second_deriv} is new to higher dimensions, and involves the second derivative of the integral of $\omega$ over lines. 

It turns out that in higher dimensions the exact problem really is too restrictive and the weights we care about for fractal sets do not satisfy condition \ref{item_Ew_psh:second_deriv}. 
In \S\ref{sec:modifying_weight_funcs} we show how to modify a weight $\omega$ to a weight $\widetilde \omega \leq \omega$ which has similar regularity but behaves better with respect to integrals over lines. Proposition \ref{prop:Ew_is_psh} can be applied to $\widetilde \omega$, and this solves \ref{BM_higher_dim_step:1}. Here is the precise statement,
see \S\ref{sec:modifying_weight_funcs} for the proof:

\begin{proposition}[\ref{BM_higher_dim_step:1}]\label{prop:BM_step1_prop}
Suppose that $\omega: \R^d \to \R_{\leq 0}$ satisfies the conditions (\ref{eq:higher_dim_BM_weight_zero_small_x},\ref{eq:regularity_condition_0to3},\ref{eq:higher_dim_BM_growth_cond}) of Theorem \ref{thm:higher_dim_BM} with constants $C_{\reg}$ and $C_{\gr}$. Then there exists a continuous plurisubharmonic function $u: \C^d \to \R$ satisfying
\begin{align}
    u(\mb x) &\leq \omega(\mb x) && \text{for $\mb x \in \R^d$}, \label{eq:u_less_omega_step1prop}\\ 
    u(\mb x) &= 0 && \text{for $|\mb x| \leq 2$, $\mb x \in \R^d$}, \label{eq:u_zero_step1prop}\\ 
    |u(\mb x_1) - u(\mb x_2)| &\leq C_{\Lip}|\mb x_1 - \mb x_2| && \text{for $\mb x_1, \mb x_2 \in \R^d$},\label{eq:u_Lip_step1prop} \\ 
    u(\mb x) &\leq u(\mb x + i \mb y) \leq u(\mb x) + \rho |\mb y| && \text{for $\mb x + i \mb y \in \C^d$}. \label{eq:u_bdd_y_step1prop}
\end{align}
We may take $C_{\Lip} \leq C_d C_{\reg}$ and $\rho \leq C_d\max(C_{\reg}, C_{\gr})$.
\end{proposition}

Now we turn to the analytic BM problem. H\"ormander's $L^2$ theory of the $\delbar$ equation applies just as well as in one dimension, which we establish in the next Proposition. 
\begin{proposition}[\ref{BM_higher_dim_step:2}]\label{prop:psh_to_bounded_fourier}
Let $u: \C^d \to \R$ be a plurisubharmonic function such that $u|_{\R^d} \leq 0$ and $u$ satisfies \eqref{eq:u_zero_step1prop}, \eqref{eq:u_Lip_step1prop}, \eqref{eq:u_bdd_y_step1prop} with constants $C_{\Lip}$ and $\rho$.
Then there exists an entire function $f: \C^d \to \C$ such that 
\begin{align*}
    |f(\mb x + i\mb y)| &\leq A\, e^{2\rho |\mb y|} && \text{for some $A > 0$}, \\ 
    |f(\mb x)| &\geq \frac{1}{2} && \text{for all $\mb x \in \B_{r_{\min}}$},\\ 
    |f(\mb x)| &\leq C\, e^{u(\mb x)}&& \text{for $\mb x \in \R^d$}, \\ 
    \int_{\R^d} |f(\mb x)|^2\, d\mb x &< \infty.
\end{align*}
We may take
\begin{align*}
    r_{\min} &= c_d\, \min(\rho, \rho^{-1}), \\ 
    C &= C_d\, e^{C_{\Lip}}\, \max(\rho^{-C_d}, e^{2\rho}).
\end{align*}
\end{proposition}
We prove Proposition \ref{prop:psh_to_bounded_fourier} in \S\ref{sec:constr_entire_from_psh} following Bourgain's one dimensional argument.

Now we can finish the proof of our higher dimensional Beurling--Malliavin multiplier theorem by combining our solutions to \ref{BM_higher_dim_step:1} and \ref{BM_higher_dim_step:2}. 

\begin{proof}[Proof of Theorem \ref{thm:higher_dim_BM}]
Let $\omega: \R^d \to \R_{\leq 0}$ satisfy the conditions of Theorem \ref{thm:higher_dim_BM} with constants $C_{\reg}$ and $C_{\gr}$, and let $\sigma > 0$ be a parameter. 

By Proposition \ref{prop:BM_step1_prop}, there exists a plurisubharmonic function $u: \C^d \to \R$ which satisfies the conditions of Proposition \ref{prop:psh_to_bounded_fourier} with constants $C_{\Lip}$ and $\rho$ where
\begin{equation*}
    C_* = \max(C_{\Lip}, \rho) \leq C_d \max(C_{\reg}, C_{\gr}).
\end{equation*}
Let $u_{\sigma} = \frac{\sigma}{C_*} u$. 
The conditions of Proposition \ref{prop:psh_to_bounded_fourier} are satisfied for $u_{\sigma}$ with 
\begin{align*}
    C_{\Lip}' = \rho' = \sigma,
\end{align*}
so applying that Proposition in conjunction with the Paley--Wiener criterion (Theorem~\ref{thm:paley_wiener}), we see that there exists $f \in L^2(\R^d)$ with $\supp \hat f \subset \B_{\sigma/\pi} \subset \B_{\sigma}$ and satisfying the necessary estimates with $c = 1/C_*$. 
\end{proof}

\section{Exact plurisubharmonic extensions}\label{sec:constr_psh_Eomega_C}
The one dimensional approach to the \ref{BM_1D_exact_step:1} problem is to harmonically extend $\omega$ to each half plane separately. Given $\omega: \R \to \R$, the Poisson kernel gives an explicit solution:
\begin{align}
    E\omega(x+iy) &= \int_{-\infty}^{\infty} \omega(x+t) P_y(t)\, dt\quad &&P_y(t) = \frac{1}{\pi} \frac{|y|}{y^2+t^2}, \nonumber\\ 
    &= \int_{-\infty}^{\infty} \frac{\omega(x+ty)}{1+t^2}\, \frac{dt}{\pi} &&\text{by change of variables}. \label{eq:poisson_ext_1D}
\end{align}

In higher dimensions we define an extension operator taking functions on $\R^d$ to functions on $\C^d$ by 
\begin{equation}\label{eq:defn_of_extension_op}
    E\omega(\mb x + i \mb y) = \int_{-\infty}^{\infty} \frac{\omega(\mb x + t \mb y)}{1+t^2}\, \frac{dt}{\pi}.
\end{equation}
If $\omega$ is Lipschitz and satisfies the growth condition \eqref{eq:higher_dim_BM_growth_cond} then the integral is finite, see Lemma \ref{lem:ext_lip_func_with_growth}.
The operator $\omega \to E\omega$ separately harmonically extends $\omega$ to every real-linear complex line
\begin{equation*}
    \ell_{\mb x, \mb y} = \{\mb x + z\mb y\, :\, z \in \C\}\subset \C^d,\quad (\mb x, \mb y) \in \R^d \times (\R^d \setminus \{0\}).
\end{equation*}
Equivalently, $E\omega$ is the unique bounded solution to the PDE 
\begin{equation}\label{eq:pde_for_Eomega}
\begin{cases}
    \langle (\del \delbar E\omega)(\mb x + i \mb y)\mb y, \mb y\rangle = 0 & \text{for } \mb x + i \mb y \in \C^d \setminus \R^d, \\ 
    E\omega(\mb x) = \omega(\mb x) & \text{for } \mb x \in \R^d.
\end{cases}
\end{equation}
It is not obvious at first that all these separate harmonic extensions combine to give a nice global extension, but equation \eqref{eq:defn_of_extension_op} shows that they do. 

Given a weight $\omega$, 
\begin{equation}
    u = E\omega + C|\mb y|
\end{equation}
will be our candidate plurisubharmonic function. Unlike the one dimensional case $E\omega$ is not plurisubharmonic away from $\R^d$, so adding the term $C|\mb y|$ will have to both make $u$ satisfy the sub-mean value property \eqref{eq:sub_mean_value_prop_psh} on complex disks centered at points of $\R^d$ and points off of $\R^d$.
Analyzing this equation leads to the following proposition which solves the \ref{BM_higher_dim_exact_step:1} problem. For $\ell = \{\mb x + t\hat{\mb y}\, :\, t \in \R\}$ a line in $\R^d$, let $\omega|_{\ell}(t) = \omega(\mb x + t \hat{\mb y})$ be $\omega$ restricted to $\ell$ (this function just depends on the line itself up to translation and reflection). 

\begin{proposition}[\ref{BM_higher_dim_exact_step:1}]\label{prop:Ew_is_psh}
Let $\omega: \R^d \to \R_{\leq 0}$ be a $C^2$ and compactly supported function satisfying 

\begin{enumerate}[leftmargin=*,labelindent=
\parindent,=itemsep=0pt,label=(\roman*), ref=(\roman*)]
    \item 
    For every line $\ell \subset \R^d$, 
    \begin{equation}\label{eq:Ew_psh_hilbert_deriv_bdd}
        \| H[\omega|_{\ell}'] \|_{\infty} \leq C_1.
    \end{equation}
    \label{item_Ew_psh:hilbert}
    \item For every $\mb x \in \R^d$, $\hat{\mb y}$ a unit vector, and $\hat {\mb v}$ a unit vector with $\mbhat y \perp \mbhat v$,
    \begin{equation}\label{eq:second_deriv_int_over_line}
        \int_{-\infty}^{\infty} \langle (D^2 \omega(\mb x + t \mbhat y))\mbhat v, \mbhat v\rangle \, \frac{dt}{\pi} \geq -C_2.
    \end{equation}
    \label{item_Ew_psh:second_deriv}
\end{enumerate}

If $C \geq \max(C_1, C_2)$, then
\begin{equation*}
    u (\mb x + i \mb y) = E\omega(\mb x + i \mb y) + C| \mb y |
\end{equation*}
is plurisubharmonic on $\C^d$ and continuous. We have 
\begin{equation}\label{eq:u_bounds_both_dir_Ew_psh}
    u(\mb x) \leq u(\mb x + i \mb y) \leq u(\mb x) + 2C|\mb y|. 
\end{equation}

\end{proposition}
Condition \ref{item_Ew_psh:hilbert} implies that $E\omega + C|\mb y|$ satisfies the sub-mean value property for complex disks centered at points of $\R^d$. Condition \ref{item_Ew_psh:second_deriv} is new to higher dimensions and it implies that $E\omega + C|\mb y|$ is plurisubharmonic on $\C^d \setminus \R^d$. 

\begin{remarks*}
\begin{enumerate}[leftmargin=*,label=\arabic*.]
    \item It turns out that in $d \geq 2$ condition \ref{item_Ew_psh:hilbert} essentially follows from condition \ref{item_Ew_psh:second_deriv}. This observation is due to Semyon Dyatlov, see \S\ref{subsec:hilbert_from_second_deriv}.
    \item Proposition \ref{prop:Ew_is_psh} is strong enough to prove Proposition \ref{prop:BM_step1_prop} in the special case of radial weights $\omega(\mb x) = f(|\mb x|)$. 
\end{enumerate}
\end{remarks*}

\subsection{Basic properties of the extension operator}
Let 
\begin{equation}\label{eq:defn_E_R}
    E_R\omega(\mb x + i \mb y) = \int_{|t| \leq R} \frac{\omega(\mb x + t \mb y)}{1+t^2}\, \frac{dt}{\pi}
\end{equation}
be the partial integral for $E\omega$. 
\begin{lemma}\label{lem:ext_lip_func_with_growth}
Let $\omega: \R^d \to \R$ be Lipschitz with constant $C_{\Lip}$ and satisfy the growth condition \eqref{eq:higher_dim_BM_growth_cond} with constant $C_{\gr}$. Then the integral defining $E\omega$ is absolutely convergent and $E_R\omega \to E\omega$ uniformly on compact subsets.
\end{lemma}
\begin{proof}
First of all, 
\begin{align*}
    \int \frac{|\omega(\mb x + t\mb y)|}{1+t^2}\, \frac{dt}{\pi} \leq C_{\Lip}|\mb x| + \int \frac{|\omega(t\mb y)|}{1+t^2}\, \frac{dt}{\pi} < \infty
\end{align*}
using both the Lipschitz property and the growth condition.

Let $R\geq 1$ and set 
\begin{align*}
    \mathrm{\Err}_R(\mb x + i\mb y) = \int_{|t|\geq R} \frac{|\omega(\mb x + t \mb y)|}{1+t^2}\, \frac{dt}{\pi}.
\end{align*}
Using that $\omega$ is Lipschitz we have
\begin{align*}
    \mathrm{\Err}_R(\mb x + i\mb y) \leq \frac{2C_{\Lip}|\mb x|}{R} + \int_{|t|\geq R} \frac{|\omega(t \mb y)|}{1+t^2}\, \frac{dt}{\pi}.
\end{align*}
We have 
\begin{align*}
    \int_{|t|\geq R} \frac{|\omega(t \mb y)|}{1+t^2}\, \frac{dt}{\pi} &= \int_{|t|\geq R} \frac{|\omega(t \mb y)|}{1+t^2} 1_{|t\mb y| \leq 1}\, dt + \int_{|t|\geq R} \frac{|\omega(t \mb y)|}{1+t^2} 1_{|t\mb y| \geq 1}\, dt.
\end{align*}
The first term is $\lesssim (|\omega(0)| + C_{\Lip})/R$. For the second term we replace $R$ by $\widetilde R = \max(R, 1/|\mb y|)$ and estimate
\begin{align*}
    \int_{|t|\geq \widetilde R} \frac{|\omega(t \mb y)|}{1+t^2} \, dt &\leq \int_{|t|\geq \widetilde R} \frac{|\omega(t \mb y)|}{t^2} \, dt \lesssim \int_{|t|\geq \widetilde R} \int_{1/2}^2\frac{|\omega(t \mb y)|}{(t/r)^2} \, drdt \\ 
    &\lesssim \int_{|t|\geq \widetilde R} \int_{1/2}^2\frac{|\omega(r t \mb y)|}{t^2} \, drdt \lesssim \int_{\widetilde R}^{\infty}\frac{G^*(|t \mb y|)}{t^2} \, dt \\ 
    &\lesssim \int_{\widetilde R|\mb y|}^{\infty}|\mb y| \frac{G^*(t)}{t^2} \, dt \lesssim |\mb y| \int_{\widetilde R|\mb y|}^{\infty} \frac{G^*(t)}{1+t^2}\, dt \\ 
    &\lesssim R^{-1/2} \int_{0}^{\infty} \frac{G^*(t)}{1+t^2}\, dt + |\mb y| \int_{R^{1/2}}^{\infty} \frac{G^*(t)}{1+t^2}\, dt
\end{align*}
where in the last line we split into the two cases $|\mb y| \leq R^{-1/2}$ and $|\mb y| \geq R^{-1/2}$. 
Suppose that $|\mb y|, |\mb x| \leq M$. Then combining our estimates, 
\begin{align*}
    \mathrm{\Err}_R(\mb x + i\mb y) \lesssim \frac{C_{\Lip}M}{R} + \frac{|\omega(0)| + C_{\Lip}}{R} + R^{-1/2} C_{\gr} + M \int_{R^{1/2}}^{\infty} \frac{G^*(t)}{1+t^2}\, dt.
\end{align*}
The right hand side goes to zero as $R \to \infty$ so $\Err_R(\mb x + i \mb y)$ goes to zero uniformly in compact subsets. It follows that $E_R\omega \to E\omega$ uniformly on compact subsets. 
\end{proof}

Next we prove our earlier claim that the Dirichlet-to-Neumann operator of $\C \setminus \R$ is $\omega \to H[-\omega']$.
\begin{lemma}\label{lem:normal_derivative_harmonic_extension}
Let $\omega \in C^1_0(\R)$ and let $u = E\omega$ be the bounded harmonic extension of $\omega$ to the upper half plane $\H$. Then 
\begin{equation}
    \partial_y u(x+i0) = H[-\omega']. 
\end{equation}
\end{lemma}
\begin{proof}
Because $u$ is harmonic on the upper half plane $\H$ and $u(x+iy) \to 0$ as $y \to \infty$, we can write $u = \Re f$ where $f = u + iv$ is analytic on $\H$ and $f(x+iy) \to 0$ as $y \to \infty$. By the Cauchy-Riemann equations,
\begin{equation}
    \partial_y u = -\partial_x v. 
\end{equation}
For fixed $y > 0$, let $u_y(x) = u(x+iy)$ and $v_y(x) = v(x+iy)$ be functions on $\R$. By the complex analytic characterization of the Hilbert transform, 
\begin{equation}
    v_y = H[u_y] \quad \text{for all $y > 0$}.
\end{equation}
Thus
\begin{equation}
    \partial_y u (x+i\varepsilon) = -\partial_x v_{\varepsilon} = H[-u_{\varepsilon}'](x)\quad \text{for all $\varepsilon > 0$}.
\end{equation}
We have $u_{\varepsilon} \to u$ in $C^1$ as $\varepsilon \to 0$, so taking a limit gives the result. 
\end{proof}
Now we establish some basic properties of $\omega \to E\omega$. 
\begin{lemma}\label{lem:props_Eomega}
Suppose $\omega\in C^2_0(\R^d)$. Then 
\begin{enumerate}[leftmargin=*,labelindent=
\parindent,=itemsep=0pt,label=(\alph*), ref=(\alph*)]
    \item $E\omega$ is $C^2$ on $\C^d \setminus \R^d$. \label{lemitem:C2_prop}
    \item For $\mb x + i \mb y \in \C^d \setminus \R^d$, let $\ell_{\mb x, \mb y} = \{\mb x + t \mbhat y\, :\, t \in \R\}$. We have
    \begin{equation}\label{eq:continuity_using_hilbert_of_deriv}
        |E\omega(\mb x + i \mb y) - E\omega(\mb x)| \leq |\mb y|\, \| H[\omega|_{\ell_{\mb x, \mb y}}']\|_{\infty} . 
    \end{equation} \label{lemitem:hilbert_derivative_continuity}
    \item Let $\mb x + i \mb y \in \C^d \setminus \R^d$. The Hermitian form $\del \delbar E\omega(\mb x + i \mb y)$ has real coefficients. Let $\mb v \in \R^d$ be given by 
    \begin{equation}\label{eq:split_v_v1_y}
        \mb v = \mb v_1 + r \mbhat y,\quad \mb v_1 \perp \mbhat y. 
    \end{equation}
    Then 
    \begin{equation}
    \begin{split}
                \langle (\del \delbar E \omega(\mb x + i \mb y)) \mb v, \mb v\rangle &=\langle (\del \delbar E \omega(\mb x + i \mb y)) \mb v_1, \mb v_1\rangle \\ 
                &= \frac{|\mb v_1|^2}{4|\mb y|} \int_{-\infty}^{\infty} \langle (D^2 \omega(\mb x + t \mbhat y)) \mbhat v_1, \mbhat v_1\rangle\, \frac{dt}{\pi}. 
    \end{split}
    \label{eq:deldelbar_Eomega}
    \end{equation}
    \label{lemitem:deldelbar_Eomega}
\end{enumerate}
\end{lemma}
\begin{proof}[Proof of \ref{lemitem:C2_prop}]
Differentiate under the integral sign in \eqref{eq:defn_of_extension_op}. 
\end{proof}

\begin{proof}[Proof of \ref{lemitem:hilbert_derivative_continuity}]

First we show \eqref{eq:continuity_using_hilbert_of_deriv} in $d = 1$. Let $\omega \in C^2_0(\R)$, and let $u = E \omega $ be the harmonic extension to $\H$. We have
\begin{equation*}
    u(x+iy) \leq u(x) + y \sup_{z \in \H} \partial_y u(z). 
\end{equation*}
The function $\partial_y u$ is harmonic on $\H$, so by the maximum principle
\begin{equation*}
    \sup_{z \in \C} \partial_y u(z) = \sup_{x \in \R} \partial_y u(x+i0). 
\end{equation*}
By Lemma \ref{lem:normal_derivative_harmonic_extension} we have $\sup_{x \in \R} \partial_y u(x+i0) \leq \| H[\omega'] \|_{\infty}$. Thus $u(x+iy) \leq u(x) + y \| H[\omega'] \|_{\infty}$. The same argument shows $u(x+iy) \geq u(x) - y \| H[\omega'] \|_{\infty}$.

Now let $\mb x + i \mbhat y \in \C^d \setminus \R^d$, $\mbhat y$ a unit vector. Let 
\begin{equation*}
    u(z) = E\omega(\mb x + z \mbhat y ),\quad z \in \C. 
\end{equation*}
Then $u(z)$ harmonically extends $\omega|_{\ell}(t) = \omega(\mb x + t \mbhat y)$, so 
\begin{equation*}
    |E\omega(\mb x + i r \mbhat y) - E\omega(\mb x)| =  |u(ir) - u(0)| \leq r\, \| H[\omega|_{\ell}'] \|_{\infty}.
\end{equation*}
\end{proof}

\begin{proof}[Proof of \ref{lemitem:deldelbar_Eomega}]
First of all, $\del \delbar E\omega$ has real coefficients because
\begin{align*}
    \Im \del_{z_j} \delbar_{z_k} E\omega &= \frac{1}{4}(\partial_{x_j} \partial_{y_k} - \partial_{x_k} \partial_{y_j}) E\omega \\ 
    &= \frac{1}{4}\int \frac{t (\partial_j \partial_k\omega)(\mb x + t \mb y) - t (\partial_k \partial_j\omega)(\mb x + t \mb y)}{1+t^2}\, \frac{dt}{\pi} = 0.
\end{align*}
It follows that $\del \delbar E\omega = \frac{1}{4}(D^2_x + D^2_y) E\omega$. 
For any $\mb x + i \mb y \in \C^d \setminus \R^d$ we have
\begin{align*}
    \del \delbar E\omega(\mb x + i \mb y) &= \frac{1}{4} (D_x^2+D_y^2) \int_{-\infty}^{\infty} \frac{\omega(\mb x + t \mb y)}{1+t^2}\, \frac{dt}{\pi} \\ 
    &= \frac{1}{4} \int_{-\infty}^{\infty} D^2\omega(\mb x + t \mb y)\, \frac{dt}{\pi}. 
\end{align*}
It is nice in this computation that the differentiation on $\mb x$ and $\mb y$ combine to give a $1+t^2$ factor, cancelling the $\frac{1}{1+t^2}$ factor in the Poisson measure. Notice $\del \delbar E\omega(\mb x, \mb y) = \frac{1}{|\mb y|} \del \delbar E\omega(\mb x, \mbhat y)$ by change of variables. Also notice that if $\mb v_1 \perp \mbhat y$ then 
\begin{equation*}
    \langle (\del \delbar E\omega(\mb x + i \mb y))\mb v_1, \mb v_1\rangle = \frac{|\mb v_1|^2}{4|\mb y|} \int_{-\infty}^{\infty} \langle (D^2\omega(\mb x + t \mbhat y))\mbhat v_1, \mbhat v_1\rangle\, \frac{dt}{\pi}
\end{equation*}
and \eqref{eq:deldelbar_Eomega} holds in this special case. To prove \eqref{eq:deldelbar_Eomega} in general, we will show that if $\mb v = \mb v_1 + r \mb y$ as in \eqref{eq:split_v_v1_y} then
\begin{equation}
    \langle (\del \delbar E\omega(\mb x + i \mb y))\mb v, \mb v\rangle = \langle (\del \delbar E\omega(\mb x + i \mb y))\mb v_1, \mb v_1\rangle. \label{eq:deldelbar_dep_v1}
\end{equation}
Define the X-ray transform by 
\begin{align*}
    X(f)(\mb x, \mbhat y) = \int_{-\infty}^{\infty} f(\mb x + t \mbhat y)\, dt,\quad (\mb x, \mbhat y) \in \R^d \times \mbS^{d-1}.
\end{align*}
We have
\begin{align*}
    \del \delbar E\omega(\mb x + i \mbhat y) = \frac{1}{4}X(D^2 \omega)(\mb x, \mbhat y) = \frac{1}{4} D_x^2 (X\omega)(\mb x, \mbhat y).
\end{align*}
The X-ray transform is constant along lines, meaning $X\omega(\mb x + a \mbhat y, \mbhat y) = X\omega(\mb x, \mbhat y)$. It follows from this property that
\begin{align*}
    D_x^2(X \omega)(\mb x, \mbhat y)\mbhat y = 0
\end{align*}
where $D_x^2 (X\omega)$ is viewed as a linear map. Equation \eqref{eq:deldelbar_dep_v1} follows.
\end{proof}

The following lemma isn't used in the proof of Proposition \ref{prop:Ew_is_psh}, but will be used in the application of Proposition \ref{prop:Ew_is_psh} to Proposition \ref{prop:BM_step1_prop}. Note that this Lemma is not necessary for the application to fractal uncertainty because the weights we construct to prove Theorem \ref{thm:FUP_higher_dim} are compactly supported. 

\begin{lemma}\label{lem:Ew_conv_cpct_subsets}
Let $\omega_j \in C(\R^d)$, $j\geq 1$ be a sequence converging to $\omega \in C(\R^d)$ uniformly on compact subsets. Suppose $\{\omega_j\}$ is uniformly Lipschitz,
\begin{equation}
    |\omega_j(\mb x_1) - \omega_j(\mb x_2)| \leq C_{\Lip}|\mb x_1 - \mb x_2|\quad \text{for all $j\geq 1$, all $\mb x_1, \mb x_2 \in \R^d$},
\end{equation}
and satisfies the uniform growth condition
\begin{equation}
\begin{split}
        G^*(r) &= \sup_{j\geq 1} \sup_{|\mb y| = r} \int_{1/2}^2 |\omega_j(s \mb y)|\, ds, \\ 
    \int_0^{\infty} \frac{G^*(r)}{1+r^2}\, dr &< \infty.
\end{split}
\label{eq:uniform_growth_cond_conv_cpct}
\end{equation}
Then $E\omega_j \to E\omega$ uniformly on compact subsets. 
\end{lemma}
\begin{proof}
Let 
\begin{equation*}
    G_1^*(r) =  \sup_{|\mb y| = r} \int_{1/2}^2 |\omega(s \mb y)|\, ds
\end{equation*}
be the growth function of $\omega$. 
Because this integral is over a compact region, $G_1^*(r) \leq G^*(r)$. Let 
\begin{align*}
    E_R \omega_j(\mb x + i \mb y) = \int_{|t|\leq R} \frac{\omega_j(\mb x + t \mb y)|}{1+t^2}\, \frac{dt}{\pi}
\end{align*}
and similarly for $\omega$. Let $\varepsilon > 0$ and $M > 0$ be given. By Lemma \ref{lem:ext_lip_func_with_growth}, if $R \geq R_0(\varepsilon, M)$ then for all $|\mb x|, |\mb y| \leq M$ we have
\begin{align*}
    |E_R\omega_j(\mb x + i \mb y) - E\omega_j(\mb x + i \mb y)| &\leq \varepsilon, \\ 
    |E_R\omega(\mb x + i \mb y) - E\omega(\mb x + i \mb y)| &\leq \varepsilon.
\end{align*}
If $j\geq j_0(\varepsilon)$ then for all $|\mb x|, |\mb y| \leq M$ we have. 
\begin{align*}
    |E_R\omega_j(\mb x + i \mb y) - E_R\omega(\mb x + i \mb y)| \leq \varepsilon. 
\end{align*}
Combining these we see $|E\omega_j(\mb x + i \mb y) - E\omega(\mb x + i \mb y)| \leq \varepsilon$ in the same region.
\end{proof}

\subsection{Proof of Proposition \ref{prop:Ew_is_psh}}\label{subsec:pf_prop_Ew_psh}
Let $U \subset \C^d$ be an open set. In \S\ref{sec:BM_problem_intro} we defined a function $u: U \to \R$ to be plurisubharmonic if it is upper semicontinuous and every restriction to a complex line is subharmonic, meaning the Laplacian is non-negative in the distributional sense. An equivalent condition is that $u$ is upper semicontinuous and satisfies the sub-mean value property
\begin{equation}\label{eq:psh_rad_dep_z}
    u(\mb z) \leq \fint_0^{2\pi} u(\mb z + e^{i\theta}\mb v)\, d\theta \quad \text{ for all $|\mb v| < r_0(\mb z)$}
\end{equation}
where $r_0(\mb z) > 0$ may depend arbitrarily on $\mb z$. For a proof see~\cite{HormanderVol1}*{Theorem 4.1.11}.  

The upshot is that the proof of Proposition \ref{prop:Ew_is_psh} can be split into two parts. Let $C > \max(C_1, C_2)$. We show $u = E\omega + C|\mb y|$
is plurisubharmonic, and it follows from continuity that we can take $C = \max(C_1, C_2)$ as well.
First we prove $u$ is plurisubharmonic on $\C^d \setminus \R^d$ using our computation of $\del \delbar E\omega$. Then we prove \eqref{eq:psh_rad_dep_z} holds for all $\mb z \in \R^d$ using our estimates on $E\omega$ near $\R^d$ \eqref{eq:continuity_using_hilbert_of_deriv}. It is in this step that we use $C > \max(C_1, C_2)$ rather than $C \geq \max(C_1, C_2)$.

Before proving plurisubharmonicity we show \eqref{eq:u_bounds_both_dir_Ew_psh}. By \eqref{eq:continuity_using_hilbert_of_deriv} we have
\begin{align*}
    \omega(\mb x) - C_1 |\mb y| \leq E\omega(\mb x + i \mb y) \leq \omega(\mb x) + C_1 |\mb y|
\end{align*}
and because $C \geq C_1$, we have
\begin{equation*}
    \omega(\mb x) \leq E\omega(\mb x + i \mb y) + C|\mb y| \leq \omega(\mb x) + 2C|\mb y|
\end{equation*}
as desired.

\subsubsection{Plurisubharmonicity on $\C^d \setminus \R^d$}
We start with a Lemma. 
\begin{lemma}\label{lem:C2_then_psh_deldelbar}
Let $U \subset \C^d$ be an open set. If $v \in C^2(U)$, then $v$ is pluri\-subharmonic on $U$ if and only if $\del \delbar v(\mb z)$ is positive semidefinite for every $\mb z \in U$. 
\end{lemma}
See \cite{HormanderConvexity}*{Corollary 4.1.5} for a proof.
\medskip

By Lemma \ref{lem:props_Eomega}\ref{lemitem:C2_prop}, $u \in C^2(\C^d \setminus \R^d)$, so it suffices to show $\del \delbar u(\mb x + i \mb y)$ is positive semidefinite for all $\mb x + i \mb y \in \C^d \setminus \R^d$ in order to establish \eqref{eq:psh_rad_dep_z} on $\C^d \setminus \R^d$. 
For $\mb y \neq 0$ we have 
\begin{equation}\label{eq:deldelbar_norm_y}
    \del \delbar |\mb y| = \frac{1}{4|\mb y|}(I - \hat {\mb y} \hat {\mb y}^t) = \frac{1}{4|\mb y|} \pi^{\perp}_{\mb y}
\end{equation}
as a Hermitian form. That is, $\del \delbar |\mb y|$ orthogonally projects away from $\mb y$ and then scales by $\frac{1}{4|\mb y|}$. If $\mb v = \mb v_1 + r \mb y$ with $\mb v_1 \perp \mb y$, then
\begin{equation}\label{eq:deldelbar_norm_y_quadr_form}
    \langle (\del \delbar |\mb y|)\mb v, \mb v\rangle = \frac{|\mb v_1|^2}{4|\mb y|}.
\end{equation}
Because $\del\delbar E\omega$ is real-linear, the goal is to show that for any $\mb v \in \R^d$, 
\begin{equation}\label{eq:psh_away_from_Rd}
    \langle \del \delbar (E\omega + C|\mb y|)\mb v, \mb v\rangle \geq 0. 
\end{equation}
Write $\mb v = \mb v_1 + r\mb y$, $\mb v_1 \perp \mb y$. Combining \eqref{eq:deldelbar_Eomega} and \eqref{eq:deldelbar_norm_y_quadr_form}, we have 
\begin{align*}
    \langle (\del \delbar E\omega + C|\mb y|)\mb v, \mb v\rangle &= \frac{|\mb v_1|^2}{4|\mb y|}\int_{-\infty}^{\infty} \langle (D^2 \omega(\mb x + t\mbhat y))\mbhat v_1, \mbhat v_1\rangle\, \frac{dt}{\pi} + C\frac{|\mb v_1|^2}{4|\mb y|} \\ 
    &= \frac{|\mb v_1|^2}{4|\mb y|}\left(C + \int_{-\infty}^{\infty} \langle (D^2 \omega(\mb x + t\mbhat y))\mbhat v_1, \mbhat v_1\rangle\, \frac{dt}{\pi}\right)
\end{align*}
so if $C \geq C_2$ then \eqref{eq:psh_away_from_Rd} holds.

\subsubsection{Plurisubharmonicity on $\R^d$}

This part is analogous to the 1D argument. Let $C \geq C_1 + \varepsilon$. 
Let $\mb x \in \R^d$, $\mb v \in \C^d \setminus \{0\}$.
By \eqref{eq:continuity_using_hilbert_of_deriv}, 
\begin{equation}
    u(\mb x + e^{i\theta}\mb v) \geq \omega(\mb x + \Re (e^{i\theta}\mb v)) + \varepsilon |\Im (e^{i\theta} \mb v)|.
\end{equation}
Because $\omega \in C_0^2(\R^d)$ there is a constant $\lambda > 0$ so that
\begin{equation*}
    |\omega(\mb x + \mb h) - \omega(\mb x) - \nabla \omega(\mb x) \cdot \mb h| \leq \lambda |\mb h|^2
\end{equation*}
for all $\mb x, \mb h \in \R^d$. Integrating, we find
\begin{equation*}
    \fint_0^{2\pi} \omega(\mb x + \Re (e^{i\theta} \mb v))\, d\theta \geq \omega(\mb x) - \lambda |\mb v|^2
\end{equation*}
for all $\mb x \in \R^d$, $\mb v \in \C^d$. On the other hand, 
\begin{align*}
    \fint_0^{2\pi} |\Im (e^{i\theta} \mb v)|\, d\theta 
    &\geq |\mb v|\fint_0^{2\pi} |\Im (e^{i\theta}\mbhat v)|^2\, d\theta \\
    &= \frac{|\mb v|}{2} \fint_0^{2\pi} (|\Im (e^{i\theta}\mbhat v)|^2+|\Re (e^{i\theta}\mbhat v)|^2)\, d\theta = \frac{|\mb v|}{2},
\end{align*}
so
\begin{align*}
    \fint_0^{2\pi} u(\mb x + e^{i\theta}\mb v)\, d\theta &\geq \fint_0^{2\pi} \omega(\mb x + \Re(e^{i\theta}\mb v))\, d\theta  + \varepsilon \fint_0^{2\pi} |\Im(e^{i\theta}\mb v)|\, d\theta \\
    &\geq \omega(\mb x) - \lambda |\mb v|^2 + \frac{\varepsilon}{2} |\mb v|. 
\end{align*}
If $|\mb v| \leq \frac{\varepsilon}{2\lambda}$ then the sub-mean value property holds.

\section{Modifying weight functions}\label{sec:modifying_weight_funcs}

In this section we prove Proposition \ref{prop:BM_step1_prop}.
Suppose $\omega: \R^d \to \R^d_{\leq 0}$ satisfies the hypotheses of Theorem \ref{thm:higher_dim_BM}. It would be nice if $\omega$ also satisfied the hypotheses of Proposition \ref{prop:Ew_is_psh}, because then we could complete the \ref{BM_higher_dim_step:1} problem. In general condition \ref{item_Ew_psh:hilbert} will be satisfied, but condition \ref{item_Ew_psh:second_deriv} on the integral of the second derivative over lines \eqref{eq:second_deriv_int_over_line} will not. Using regularity of $D^2 \omega$ \eqref{eq:regularity_condition_0to3} we can get a decent estimate for \eqref{eq:second_deriv_int_over_line} on each dyadic scale by putting absolute values inside the integral, but these contributions will not be summable. 
To fix this issue we modify the weight $\omega$ to a new weight $\widetilde \omega \leq \omega$ which has a lot of cancellation in \eqref{eq:second_deriv_int_over_line}. In our estimates we will not put absolute values inside the integral.

An important observation is that when we zoom out far enough all lines look like they pass through the origin. To be a bit more precise, using the regularity hypothesis \eqref{eq:regularity_condition_0to3} it suffices to estimate \eqref{eq:second_deriv_int_over_line} for lines through the origin. For a function $f: \R^d \to \R$, let  
\begin{equation*}
    \pi_{\mbS^{d-1}}f(\mbhat v) = \int_0^{\infty} f(t\mbhat v)\, t^{-2} dt
\end{equation*}
be a weighted spherical projection of $f$. The factor $t^{-2}$ allows us to compare the translational derivative to the rotational derivative.

\begin{lemma}\label{lem:second_deriv_spherical_prj}
Suppose $f: \R^d \to \R$ is a $C^2$ function compactly supported and supported away from the origin. Let $\mbhat y \perp \mbhat v$. Then 
\begin{equation}
    \int_0^{\infty} \langle (D^2f(t\mbhat y))\mbhat v, \mbhat v\rangle\, dt = \pi_{\mbS^{d-1}} f(\mbhat y) + \frac{d^2}{d\theta^2}\Big|_{\theta = 0} \pi_{\mbS^{d-1}} f(\mbhat y \cos \theta + \mbhat v \sin \theta  ).
\end{equation}
\end{lemma}
\begin{proof}
We have 
\begin{align*}
    \frac{d^2}{d\theta^2}\Big|_{\theta = 0} f(t\mbhat y \cos \theta + t\mbhat v \sin \theta) &= t^2 \langle (D^2f(t\mbhat y))\mbhat v, \mbhat v\rangle - t (\partial_{\mbhat y}f)(t\mbhat y).
\end{align*}
Integrating,
\begin{align*}
    \frac{d^2}{d\theta^2}\Big|_{\theta=0} \pi_{\mbS^{d-1}} f(\mbhat y \cos \theta + \mbhat v \sin \theta) &= \int_0^{\infty} \frac{d^2}{d\theta^2}\Big|_{\theta=0} f(t\mbhat y \cos \theta + t\mbhat v \sin \theta)\, t^{-2}dt \\ 
    &= \int_0^{\infty} \langle (D^2f(t\mbhat y))\mbhat v, \mbhat v\rangle\, dt - \int_0^{\infty} t^{-1}\frac{d}{dt}f(t\mbhat y)\, dt \\ 
    &= \int_0^{\infty}\langle (D^2f(t\mbhat y))\mbhat v, \mbhat v\rangle\, dt  - \int_0^{\infty} f(t\mbhat y)\, t^{-2}dt  
\end{align*}
using integration by parts in the last step.
\end{proof}

We write the modified weight as a sum of dyadic pieces, $\widetilde \omega = \sum_k \widetilde \omega_k$. The idea is to design each piece $\widetilde \omega_k$ so that $\pi_{\mbS^{d-1}} \widetilde \omega_k \equiv q_k = \text{const}$. Lemma \ref{lem:second_deriv_spherical_prj} then gives 
\begin{equation}
    \int_0^{\infty} \langle (D^2\widetilde \omega_k(t\mbhat y))\mbhat v, \mbhat v\rangle\, dt = q_k,
\end{equation}
and as long as $\sum_k |q_k| < \infty$ we obtain a favorable estimate. 

We implement this plan with the following two Lemmas. The first Lemma modifies the weight $\omega \to \widetilde \omega$.
\begin{lemma}\label{lem:modify_weight_lemma}
Suppose that $\omega: \R^d \to \R_{\leq 0}$ satisfies (\ref{eq:higher_dim_BM_weight_zero_small_x},\ref{eq:regularity_condition_0to3},\ref{eq:higher_dim_BM_growth_cond}), the conditions of Theorem \ref{thm:higher_dim_BM}, with constants $C_{\reg}$ and $C_{\gr}$. Then there exists a weight $\widetilde \omega = \sum_{k\geq 0} \widetilde \omega_k$ such that
\begin{enumerate}[leftmargin=*,labelindent=\parindent,=itemsep=0pt,label=(\roman*), ref=(\roman*)]    
    \item $\widetilde\omega (\mb x)  \leq \omega(\mb x)$ for all $\mb x\in \R^d$,  \label{lemitem:tilde_omega_less_omega}
    \item $\widetilde \omega(\mb x) = 0$ for $|\mb x| \leq 2$, \label{lemitem:tilde_omega_0_x_small}
    \item $\supp \widetilde \omega_0 \subset \{\mb x\, :\, |\mb x| \leq 5\}$ and $\supp \widetilde \omega_k \subset \{\mb x\, :\, 2^{k-1} \leq |\mb x| \leq 2^{k+2}\}$ for $k\geq 1$, \label{lemitem:supp_cond_tilde_omega}
    \item We have
    \begin{equation}
        \|D^a \widetilde \omega_k\|_{\infty} \leq C_{\reg}' 2^{(1-a)k}\quad \text{for $0 \leq a \leq 3$},\label{eq:omega_reg_pieces}
    \end{equation} \label{lemitem:tilde_omega_regularity}
    \item For $k \geq 5$, $\pi_{\mbS^{d-1}} \widetilde \omega_k = q_k$ is constant over the unit sphere and 
    \begin{equation}
        \sum_{k} |q_k| \leq C_{\gr}'. \label{eq:sum_qk_omega_tilde}
    \end{equation} \label{lemitem:tilde_omega_spherical_proj}
\end{enumerate}
We may take $C_{\reg}' \leq C_d C_{\reg}$ and $C_{\gr}' \leq C_dC_{\gr}$.
\end{lemma}
The condition \eqref{eq:omega_reg_pieces} is just another way of writing the regularity condition \eqref{eq:regularity_condition_0to3}, and \eqref{eq:sum_qk_omega_tilde} is another way of writing the growth condition \eqref{eq:higher_dim_BM_growth_cond}.

The second lemma analyzes the modified weight $\widetilde \omega$ and shows it is admissible for Proposition \ref{prop:Ew_is_psh}.

\begin{lemma}\label{lem:modified_weight_good_lemma}
Suppose that $\widetilde \omega = \sum_{k\geq 0} \widetilde \omega_k$ satisfies the conditions Lemma \ref{lem:modify_weight_lemma}\ref{lemitem:tilde_omega_0_x_small}-\ref{lemitem:tilde_omega_spherical_proj}. 
Let $C = C_d\max(C_{\reg}', C_{\gr}')$.
Then $u = E\widetilde \omega + C|\mb y|$ is continuous and plurisubharmonic on $\C^d$ and satisfies 
\begin{equation}\label{eq:u_bdd_above_below_modified_weight}
    u(\mb x) \leq u(\mb x + i \mb y) \leq u(\mb x) + 2 C |\mb y|. 
\end{equation}
\end{lemma}

Combining these two lemmas proves Proposition \ref{prop:BM_step1_prop} and completes the \ref{BM_higher_dim_step:1} problem. Note that the Lipschitz condition \eqref{eq:u_Lip_step1prop} in the conclusion of Proposition \ref{prop:BM_step1_prop} follows from \eqref{eq:omega_reg_pieces} and the fact that $u = \widetilde \omega$ on $\R^d$. 

\subsection{Proof of Lemma \ref{lem:modify_weight_lemma}: Modifying the weight}

Let $\omega$ satisfy the conditions of Theorem \ref{thm:higher_dim_BM} with constants $C_{\reg}, C_{\gr}$. Let
\begin{equation}
    1 = \sum_{k=0}^{\infty} \psi_k
\end{equation}
be a partition of unity of $\R_{\geq 0}$ where $\supp \psi_k \subset A_k$, 
\begin{alignat}{2}
    A_0 &= [0,5], \\ 
    A_k &= [2^{k-1}, 2^{k+2}]&\qquad \text{ for $k\geq 1$}. 
\end{alignat}
We may choose $\psi_k(\mb x) = \psi_1(2^{1-k}\mb x)$ giving a derivative estimate $|D^a \psi_k(\mb x)| \leq C_a 2^{-ak}$ for all $a \geq 0$. Write $\psi_k(\mb x) = \psi_k(|\mb x|)$ for $\mb x \in \R^d$. 
Let 
\begin{equation}
    \pi_{\mbS^{d-1}}\psi_k(\mbhat v) = p_k,\quad p_k \sim 2^{-k}\text{ up to universal constants}.
\end{equation}
Write 
\begin{equation}
    \omega = \sum_{k\geq 0} \omega_k,\quad \omega_k(\mb x) = \psi_k(\mb x) \omega(\mb x). 
\end{equation}
For $k \geq 1$, let 
\begin{equation}
    q_k = \inf_{\mbhat v \in \mbS^{d-1}} \pi_{\mbS^{d-1}}\omega_k(\mbhat v).
\end{equation}
Recall that $\omega \leq 0$, so $|q_k| = \sup_{\mbhat v \in \mbS^{d-1}} \pi_{\mbS^{d-1}}|\omega_k(\mbhat v)|$. 
Now set
\begin{equation}\label{eq:defn_of_gk}
    g_k(\mb x) = p_k^{-1} \psi_k(\mb x)\, (q_k - (\pi_{\mbS^{d-1}} \omega_k)(\mbhat x)),\quad k\geq 1.  
\end{equation}
Notice that by the definition of $q_k$, we have $g_k\leq 0$.
We define 
\begin{align}
    \widetilde \omega_k &= \begin{cases}
        \omega_k & 0 \leq k < 5,  \\ 
        \omega_k + g_k & k \geq 5,
    \end{cases} \\ 
    \widetilde \omega &= \sum_{k \geq 0} \widetilde \omega_k.
\end{align}
Certainly $\widetilde \omega \leq \omega$ because $g_k \leq 0$ for all $k$. Also, because we only add the modification $g_k$ for $k\geq 5$, we have $\widetilde \omega(\mb x) = \omega(\mb x) = 0$ for $|\mb x| \leq 2$. 
By construction,
\begin{equation}\label{eq:proj_of_gk_plus_omegak_const}
    \pi_{\mbS^{d-1}}\widetilde \omega_k = q_k\quad \text{for $k\geq 5$}.
\end{equation}
We have 
\begin{align*}
    |q_k| &= \sup_{\mbhat v \in \mbS^{d-1}} \int_0^{\infty}|\omega_k(t\mbhat v)|\, t^{-2}dt \lesssim  2^{-2k} \sup_{\mbhat v \in \mbS^{d-1}} \int_{2^{k-1}}^{2^{k+2}} |\omega(t\mbhat v)|\, dt \\ 
    &\lesssim 2^{-k}(G^*(2^k) + G^*(2^{k+1})).
\end{align*}
Choose $\mb x \in \R^d$ with $|\mb x| = r$ so that $G^*(r) = G(\mb x)$. We have 
\begin{align*}
    G^*(r) = \int_{1/2}^2 |\omega(s\mb x)|\, ds  \lesssim \int_{1/2}^2 \int_{1/2}^2 |\omega(st\, \mb x)|\, ds dt \lesssim \int_{1/2}^2 G^*(tr)\, dt
\end{align*}
leading to the pointwise bound 
\begin{equation*}
    G^*(2^j) \lesssim 2^{-j}\int_{2^{j-1}}^{2^{j+1}} G^*(r)\, dr
\end{equation*}
which gives
\begin{equation*}
    \sum_k |q_k| \lesssim \sum_k 2^{-k} G(2^k) \lesssim \int_0^{\infty} \frac{G^*(r)}{1+r^2}\, dr
\end{equation*}
as needed. 

Finally, we must show that $\widetilde \omega$ satisfies the regularity condition \eqref{eq:omega_reg_pieces}. 
Let $0\leq a \leq 3$. By the Leibniz rule,
\begin{align*}
    \| D^a \omega_k \|_{\infty} &\lesssim \sum_{0\leq b \leq a} \| D^{a-b}\psi_k\|_{\infty} \sup_{|\mb x| \in A_k} |D^b\omega(\mb x)|\\
    &\lesssim C_{\reg}\sum_{0\leq b\leq a} 2^{-(a-b)k}2^{(1-b)k} \lesssim C_{\reg} 2^{(1-a)k}.
\end{align*}
Let $h_k(\mb x) = \pi_{\mbS^{d-1}}\omega_k(\mbhat x)$. We have 
\begin{align*}
    h_k(\mb x) &= \int_{2^{k-1}}^{2^{k+2}} \omega_k(t\mbhat x)\, t^{-2}dt \\ 
    &= |\mb x|^{-1} \int_{0}^{\infty} \omega_k(s\mb x)\, s^{-2}ds,\\
    g_k(\mb x) &= p_k^{-1} \psi_k(\mb x) (q_k - h_k(\mb x))
\end{align*}
Thus
\begin{align*}
    \|D^a g_k \|_{\infty} &\lesssim 2^k \sum_{0\leq b \leq a} \| D^{a-b} \psi_k\|_{\infty}\, \sup_{|\mb x| \in A_k } |D^b (q_k-h_k)(\mb x)| \\
    &\lesssim \sum_{0\leq b \leq a} 2^{-(a-b)k+k}\sup_{|\mb x| \in A_k } |D^b h_k(\mb x)|. 
\end{align*}
Let $|\mb x| \in A_k$, $0\leq b \leq 3$. We have
\begin{align*}
    |D^b h_k(\mb x)| &\lesssim \sum_{0\leq c \leq b} |D^{b-c} |\mb x|^{-1}|\, \int_{1/10}^{10}  |D^c \omega_k(s\mb x)|\, s^{c-2}\, ds \\ 
    &\lesssim C_{\reg}\sum_{0\leq c \leq b} 2^{-(1+(b-c))k} 2^{(1-c)k} \lesssim C_{\reg}2^{-bk}.
\end{align*}
Combining these estimates we obtain that for $0 \leq a \leq 3$,
\begin{align*}
    \| D^a g_k \|_{\infty} &\lesssim C_{\reg}2^{(1-a)k}\\
    \| D^a \widetilde \omega_k \|_{\infty} &\lesssim C_{\reg}2^{(1-a)k}
\end{align*}
as needed.

\subsection{Proof of Lemma \ref{lem:modified_weight_good_lemma}: Analyzing the modified weight}

We would like to apply Proposition \ref{prop:Ew_is_psh} to $\widetilde \omega$. Let $\widetilde \omega = \sum_{k\geq 0} \widetilde \omega_k$ satisfy the conditions Lemma~\ref{lem:modify_weight_lemma}\ref{lemitem:tilde_omega_0_x_small}-\ref{lemitem:tilde_omega_spherical_proj}. 
First we prove an estimate on the Hilbert transform of the derivative of $\widetilde \omega$ restricted to lines.
\begin{lemma}
Let $\ell = \{\mb x + t \mbhat y\, :\, t \in \R\}$ be a line. Let $\widetilde \omega_k|_{\ell}(t) = \widetilde \omega_k(\mb x + t \mbhat y)$ be the restriction of $\widetilde \omega_k$ to this line. For all such lines, we have
\begin{equation}
    \sum_{k\geq 0} |H[\widetilde \omega_k|_{\ell}'](0)| \lesssim C_{\reg}'+C_{\gr}' \label{eq:estimate_H_tilde_omega_k_sum}.
\end{equation}
\end{lemma}
\begin{proof}
Let $r = 0$ if $|\mb x| \leq 4$, and otherwise let $r\geq 1$ be such that $|\mb x| \in [2^{r-1}, 2^r)$. 

For any $k, r$ we have the following estimate, although we only use it when $r-5 \leq k \leq r+5$:
\begin{align*}
    | H[\widetilde \omega_k|_{\ell}'](0)| &= \left|\int_{0}^{\infty} \frac{\partial_{\mbhat y}\widetilde \omega_k(\mb x + t \mbhat y) - \partial_{\mbhat y}\widetilde \omega_k(\mb x - t \mbhat y)}{t}\, \frac{dt}{\pi}\right| \\ 
    &\lesssim 2^k \| D^2 \widetilde \omega_k \|_{\infty} \leq C'_{\reg}.
\end{align*}
For $k < r-5$, $\widetilde \omega_k$ is supported away from $\mb x$, and we have
\begin{align*}
    | H[\widetilde \omega_k|_{\ell}'](0)| &= \left|\int_{-\infty}^{\infty} \frac{1}{t}\, \frac{d}{dt}\widetilde \omega_k(\mb x+ t\mbhat y)\, \frac{dt}{\pi}\right| \\ 
    &=  \left|\int_{-\infty}^{\infty} \frac{\widetilde \omega_k(\mb x+ t\mbhat y)}{t^2}\, \frac{dt}{\pi}\right| && \text{by integration by parts,}\\ 
    &\lesssim 2^{-2r} 2^k\| \widetilde \omega_k \|_{\infty} \leq C_{\reg}'2^{2(k-r)}.
\end{align*}
Finally, for $k > r+5$, $\widetilde \omega_k$ is once again supported away from $\mb x$, and integrating by parts we have
\begin{align*}
    | H[\widetilde \omega_k|_{\ell}'](0)| &= \left|\int_{-\infty}^{\infty} \frac{\widetilde \omega_k(\mb x+ t\mbhat y)}{t^2}\, \frac{dt}{\pi}\right| \\ 
    &\lesssim 2^{-2 k} \int_{-\infty}^{\infty}|\widetilde \omega_k(\mb x + t \mbhat y)|\, dt \\ 
    &\lesssim C_{\reg}'2^{-k} |\mb x| + 2^{-2k} \int_{-\infty}^{\infty}|\widetilde \omega_k(t \mbhat y)|\, dt && \text{by Lipschitz regularity} \\ 
    &\lesssim C_{\reg}'2^{-k} |\mb x| + |q_k|.
\end{align*}
Summing these contributions, 
\begin{align*}
    \sum_{k\geq 0} | H[\widetilde \omega_k|_{\ell}'](0)| &\lesssim C_{\reg}' + C_{\reg}'\sum_{k < r-5} 2^{2(k-r)} + C_{\reg}'\sum_{k > r+5} 2^{r-k} + \sum_{k\geq 5} |q_k| \\
    &\lesssim C_{\reg}' + C_{\gr}'.
\end{align*}
\end{proof}

Now we prove an estimate on the integral of the second derivative of $\widetilde \omega$ over lines.
\begin{lemma}
Let $\ell = \{\mb x_0 + t \mbhat y\}$ be a line, where $\mb x_0$ is the closest point to the origin.  We have
\begin{equation}
    \sum_{k\geq 0} \left|\int_{-\infty}^{\infty} \langle (D^2 \widetilde \omega_k(\mb x_0 + t \mbhat y))\mbhat v, \mbhat v \rangle\, dt\right| \lesssim C_{\reg}'+C_{\gr}'\quad \text{for all $\mbhat v \perp \mbhat y$}.\label{eq:estimate_second_deriv_line_tilde_omega_k_sum}
\end{equation}
\end{lemma}
\begin{proof}
Let $\mbhat v \perp \mbhat y$. 
Let $r = 0$ if $|\mb x_0| \leq 4$, and otherwise let $r\geq 1$ be so that $|\mb x_0| \in [2^{r-1}, 2^r)$. 
For $k < r - 5$ the support of $\widetilde \omega_k$ does not intersect $\ell$ and 
\begin{align*}
    \int_{-\infty}^{\infty} \langle (D^2 \widetilde \omega_k(\mb x_0 + t \mbhat y))\mbhat v, \mbhat v\rangle\, dt = 0.
\end{align*}
For $r-5 < k < r+5$ we put the absolute values inside the integral and use the second derivative regularity condition,
\begin{align*}
    \Bigl|\int_{-\infty}^{\infty}  \langle (D^2 \widetilde \omega_k(\mb x_0 + t \mbhat y))\mbhat v, \mbhat v\rangle\, dt\Bigr| &\lesssim \int_{-\infty}^{\infty}  |D^2 \widetilde \omega_k(\mb x_0 + t \mbhat y)|\, dt \\ 
    &\lesssim 2^k \| D^2 \widetilde \omega_k \|_{\infty} \leq C_{\reg}'.
\end{align*}
Next, let $k > r+5$. We translate the integral to a line through the origin using the third derivative regularity condition,
\begin{align*}
    \left|\int_{-\infty}^{\infty} \langle (D^2 \widetilde \omega_k(\mb x_0 + t \mbhat y))\mbhat v, \mbhat v\rangle\, dt\right| \leq \left|\int_{-\infty}^{\infty} \langle (D^2 \widetilde \omega_k(t \mbhat y))\mbhat v, \mbhat v\rangle\, dt\right| + C_{\reg}'|\mb x_0|2^{-k}.
\end{align*}
By the hypothesis that $\pi_{S^{d-1}} \omega_k = q_k$ and Lemma \ref{lem:second_deriv_spherical_prj} on the second derivative of spherical projections,
\begin{align*}
    \int_{-\infty}^{\infty} \langle (D^2 \widetilde \omega_k(t \mbhat y))\mbhat v, \mbhat v\rangle\, dt = 2q_k.
\end{align*}
Thus
\begin{align*}
    \sum_k \left|\int_{-\infty}^{\infty} \langle (D^2 \widetilde \omega_k(\mb x_0 + t \mbhat y))\mbhat v, \mbhat v\rangle\, dt\right| &\lesssim C_{\reg}' + \sum_{2^k \geq |\mb x_0|} \left(C_{\reg}'|\mb x_0|2^{-k} + |q_k|\right) \\
    &\lesssim C_{\reg}' + C_{\gr}'.
\end{align*}
\end{proof}

Finally, we finish the proof of Lemma \ref{lem:modified_weight_good_lemma}. 
\begin{proof}[Proof of Lemma \ref{lem:modified_weight_good_lemma}]
Let 
\begin{equation}
    \widetilde \omega_{\leq k} = \sum_{0 \leq j \leq k} \widetilde \omega_j.
\end{equation}
By \eqref{eq:estimate_H_tilde_omega_k_sum} and \eqref{eq:estimate_second_deriv_line_tilde_omega_k_sum} the compactly supported weights $\widetilde \omega_{\leq k}$ satisfy the hypotheses of Proposition \ref{prop:Ew_is_psh} uniformly in $k$, and there is some $C \lesssim C_{\reg}' + C_{\gr}'$ such that for all $k \geq 1$,
\begin{equation*}
    u_{\leq k} = E\widetilde \omega_{\leq k} + C|\mb y|
\end{equation*}
is plurisubharmonic and satisfies 
\begin{equation*}
    u_{\leq k}(\mb x) \leq u_{\leq k} (\mb x + i \mb y) \leq u_{\leq k}(\mb x) + 2C|\mb y|.
\end{equation*}
Notice that the sequence $\{\widetilde \omega_k\}_{k=1}^{\infty}$ is uniformly Lipschitz by \eqref{eq:omega_reg_pieces}, and satisfies the uniform growth condition \eqref{eq:uniform_growth_cond_conv_cpct} because of \eqref{eq:sum_qk_omega_tilde}. By Lemma  \ref{lem:Ew_conv_cpct_subsets}, $E\widetilde \omega_{\leq k} \to E\widetilde \omega$ uniformly on compact sets. It follows that 
\begin{equation*}
    u = E\widetilde \omega + C|\mb y|
\end{equation*}
is plurisubharmonic and satisfies
\begin{equation*}
    u(\mb x) \leq u(\mb x + i \mb y) \leq u(\mb x) + 2C|\mb y|. 
\end{equation*}
\end{proof}

\section{Bounded Fourier support from plurisubharmonic functions}\label{sec:constr_entire_from_psh}
In this section we prove Proposition \ref{prop:psh_to_bounded_fourier}. We are given a plurisubharmonic function $u: \C^d \to \R$ satisfying 
\begin{align}
    u(\mb x) &\leq 0 && \text{for all $\mb x \in \R^d$}, \\ 
    u(\mb x) &= 0 && \text{for $|\mb x| \leq 2$}, \\ 
    |u(\mb x_1) - u(\mb x_2)| &\leq C_{\Lip} |\mb x_1 - \mb x_2| && \text{for all $\mb x_1, \mb x_2 \in \R^d$},  \\ 
    u(\mb x) &\leq u(\mb x + i \mb y) \leq u(\mb x) + \rho |\mb y| && \text{for all $\mb x + i \mb y \in \C^d$}.
\end{align} 
We would like to construct an entire function $f: \C^d \to \C$ satisfying 
\begin{align}
    |f(\mb x + i \mb y)|&\leq A\,e^{2\rho |\mb y|}&& \text{for some $A > 0$},\label{eq:psh_to_bdd_fourier_f_growth_y} \\ 
    |f(\mb x)| &\geq \frac{1}{2} && \text{for all $\mb x \in \B_{r_{\min}}$}, \label{eq:prop_f_lower_bd}\\ 
    |f(\mb x)| &\leq C\, e^{u(\mb x)}&& \text{for $\mb x \in \R^d$}, \label{eq:prop_f_bdd_u} \\ 
    \int_{\R^d} |f(\mb x)|^2\, d\mb x &< \infty,
\end{align}
where
\begin{align}
    r_{\min} &= c_d\, \min(\rho, \rho^{-1}), \label{eq:defn_of_rmin}\\ 
    C &= C_d\, e^{C_{\Lip}}\, \max(\rho^{-C_d}, e^{2\rho}).
\end{align}
Following Bourgain we use H\"ormander's $L^2$ theory of the $\delbar$ equation to construct $f$.
\begin{theorem}[H\"ormander \cite{HormanderDelbar}*{Theorem 2.2.1'}]\label{thm:hormander_delbar}
Let $\varphi: \C^d \to \R$ be strictly plurisubharmonic with $\del \delbar \varphi(\mb z) \geq \kappa(\mb z) > 0$. Let $\eta$ be a $(0, 1)$ form on $\C^d$ with $\delbar \eta = 0$. Suppose that 
\begin{equation*}
    \int_{\C^d} |\eta(\mb z)|^2 \frac{e^{-\varphi(\mb z)}}{\kappa(\mb z)} < \infty
\end{equation*}
where we integrate with respect to the Lebesgue measure on $\C^d$. 
Then the equation $\delbar g = \eta$ has a solution $g$ satisfying 
\begin{equation}
    \int_{\C^d} |g(\mb z)|^2 e^{-\varphi(\mb z)} \leq \int_{\C^d} |\eta(\mb z)|^2 \frac{e^{-\varphi(\mb z)}}{\kappa(\mb z)}.\label{eq:hormander_thm_conclusion}
\end{equation}
\end{theorem}
We mean $\del \delbar \varphi(\mb z) \geq \kappa(\mb z)$ in the distributional sense ($\varphi$ can be an arbitrary plurisubharmonic function). 

The first important point is that the $L^2$ bound \eqref{eq:hormander_thm_conclusion} can be converted to a pointwise bound by subharmonicity.
\begin{lemma}\label{lem:pointwise_bd_subharmonicity}
Let $U \subset \C^d$ be an open set and $f: U \to \C$ analytic. If $\B_r(\mb z) \subset U$ then
\begin{equation}
    |f(\mb z)| \leq C_d r^{-d}\, \| f \|_{L^2(\B_r(\mb z))}.
\end{equation}
\end{lemma}
\begin{proof}
Because $f$ is analytic on $U$, $|f|^2$ is plurisubharmonic and thus subharmonic on $U$. It follows that 
\begin{align*}
    |f(\mb z)|^2 \leq \fint_{\B_r(\mb z)} |f(\mb w)|^2 \leq C_d r^{-2d}\, \| f \|_{L^2(\B_r(\mb z))}^2. 
\end{align*}
\end{proof}
If we ignore \eqref{eq:prop_f_lower_bd} for a moment and just want $f$ to satisfy \eqref{eq:psh_to_bdd_fourier_f_growth_y} and \eqref{eq:prop_f_bdd_u} then we could try applying H\"ormander's theorem to construct $f$ with $\eta = 0$. That doesn't work because the solution could be $f = 0$---remember that \eqref{eq:prop_f_lower_bd} quantifies the non-vanishing of $f$. To fix this we write $f = h - g$ where $h$ is a bump function in a neighborhood of the origin and $g$ solves the inhomogenous equation $\delbar g = \delbar h$. Now we can use Theorem \ref{thm:hormander_delbar} to construct $g$. By adding a new term to the plurisubharmonic weight $u$, we can force $g$ to be small near the origin and then get a lower bound on $f$ near the origin.

\subsection{Construction of the plurisubharmonic weight \texorpdfstring{$\varphi$}{phi}}
Let $\eta_{\geq 10}$ be a bump function supported on $\{\mb x\in \R^d\, :\, |\mb x| \geq 5\}$ and which takes the value $1$ for $|\mb x| \geq 10$. Let
\begin{equation}\label{eq:defn_of_omega_0}
    \omega_0 = -\eta_{\geq 10} \frac{|\mb x|}{(\log(2+|\mb x|))^2} 
\end{equation}
Then $\omega_0$ satisfies the hypotheses of Proposition \ref{prop:BM_step1_prop} so for some $c_d > 0$, $c_d\omega_0$ has a plurisubharmonic extension $u_0: \C^d \to \R$ which satisfies
\begin{align}
    |u_0(\mb x_1) - u_0(\mb x_2)| &\leq |\mb x_1 - \mb x_2| \quad \text{for all $\mb x_1, \mb x_2 \in \R^d$},\label{eq:u0_Lip} \\ 
    u_0(\mb x) &\leq u_0(\mb x + i \mb y) \leq u_0(\mb x) + \frac{1}{2}|\mb y|.   \label{eq:u0_bdd_y}
\end{align}
Let 
\begin{equation}
    \varphi =  2u + 20d\, \log |\mb z|_{\infty} + \rho\, u_0 + \frac{\rho}{2}(\langle \mb y \rangle - 1).
\end{equation}
We add the term $20d\, \log |\mb z|_{\infty}$ to get lower bounds on $f$ near the origin, we add $\rho\, u_0$ to balance out the prior term for $\mb x \in \R^d$ far from the origin, and we add $\frac{\rho}{2}(\langle \mb y \rangle - 1)$ to make $\varphi$ strictly plurisubharmonic.
Notice that 
\begin{align*}
    \log |\mb z|_{\infty} = \max_j \log |z_j|
\end{align*}
is plurisubharmonic because it is the maximum of a collection of plurisubharmonic functions. We compute that
\begin{equation}
    \del \delbar \langle \mb y \rangle = \frac{1}{4}\langle \mb y \rangle^{-3} (1 + |\mb y|^2 - \mb y \mb y^t)
\end{equation}
as a Hermitian matrix. The minimal eigenvector is $\mbhat y$, and 
\begin{equation}
    \langle (\del \delbar \langle \mb y \rangle)\mbhat y, \mbhat y \rangle = \frac{1}{4}\langle \mb y \rangle^{-3}
\end{equation}
so $(\del\delbar \langle \mb y \rangle) (\mb x + i \mb y) \geq \frac{1}{4}\langle \mb y \rangle^{-3}$. Because the other terms in $\varphi$ are also plurisubharmonic we have
\begin{equation}
    \del \delbar \varphi(\mb x + i \mb y) \geq \kappa(\mb z)    = \frac{\rho}{8}\, \langle \mb y \rangle^{-3}  \label{eq:varphi_deldelbar_bd}.
\end{equation}

\subsection{Proof of Proposition \ref{prop:psh_to_bounded_fourier}}
Let $h$ be a smooth bump function on $\C^d$ with $h = 1$ on $\B_{1/2}$ and $\supp h \subset \B_1$. Let $\eta = \delbar h$, so $\eta$ is supported on $\B_1 \setminus \B_{1/2}$. Notice that because 
\begin{equation*}
    2u(\mb x + i \mb y) + \rho u_0(\mb x + i \mb y) + \frac{\rho}{2}(\langle \mb y \rangle - 1) \geq 2u(\mb x) + \rho u_0(\mb x) = 0 \quad \text{for $|\mb x| \leq 2$}, 
\end{equation*}
we have 
\begin{equation}
    \varphi(\mb x + i \mb y) \geq  -20d\, \log 2\sqrt{2d} \quad \text{for $1/2 \leq |\mb x + i \mb y| \leq 2$.}
\end{equation}
It follows that 
\begin{align}
    \int_{\C^d} |\eta(\mb z)|^2 \frac{e^{-\varphi(\mb z)}}{\kappa(\mb z)} \leq C_d\, \rho^{-1},
\end{align}
and by Theorem \ref{thm:hormander_delbar} there is a smooth $g$ such that $\delbar g = \eta$ and 
\begin{equation}\label{eq:hormander_bound_g}
    \int_{\C^d} |g(\mb z)|^2 e^{-\varphi(\mb z)} \leq C_d\, \rho^{-1}. 
\end{equation}
Another way to write this is that $\| g\,e^{-\varphi/2} \|_{L^2(\C^d)} \leq C_d\, \rho^{-1/2}$. Define
\begin{equation}
    f = h - g.
\end{equation}
By construction, $\delbar f = 0$ so $f$ is entire. 

First we prove some upper bounds on $g$ near the origin. Because $\eta = 0$ on $\B_{1/2}$, $g$ is analytic on $\B_{1/2}$. Let $\mb x \in \R^d$ with $|\mb x| \leq 1/4$. Applying Lemma \ref{lem:pointwise_bd_subharmonicity} to $g$ with $r = |\mb x|$ we obtain
\begin{align*}
    |g(\mb x)| \leq C_d |\mb x|^{-d}\,  \| g\, e^{-\varphi/2}\|_{L^2(\B_{2|\mb x|})}\, \| e^{\varphi/2} \|_{L^{\infty}(\B_{2|\mb x|})}.
\end{align*}
We have 
\begin{equation*}
    \varphi(\mb x + i \mb y) \leq 3\rho |\mb y| + 20d \log |\mb x + i \mb y|_{\infty}\quad \text{for all $\mb x +i \mb y \in \C^d$},
\end{equation*}
so combining this with the $L^2$ estimate \eqref{eq:hormander_bound_g} we find
\begin{equation}
    |g(\mb x)| \leq C_d \rho^{-1/2}\, |\mb x|^{9d}\, e^{3\rho |\mb x|}\quad \text{for $|\mb x| \leq 1/4$}.\label{eq:bound_g_x_small}
\end{equation}
Thus $|g(\mb x)| \leq 1/2$ when $|\mb x| \leq r_{\min}$, as long as the constant $c_d$ in \eqref{eq:defn_of_rmin} is small enough. So $|f(\mb x)| \geq 1/2$ when $|\mb x| \leq r_{\min}$.

Now we prove that $f(\mb x)$ decays like $e^{u(\mb x)}$ on $\R^d$. To deal with the fact that $g$ is not analytic on $\B_1 \setminus \B_{1/2}$ we prove an $L^2$ bound for $f$ on the open set $U = \{\mb z \in \C^d \, :\, |\mb z| > 1/8\}$. We have 
\begin{equation}\label{eq:L2_bd_f}
    \| f\, e^{-\varphi/2} \|_{L^2(U)} \leq \| h\,e^{-\varphi/2} \|_{L^2(U)} + \| g\,e^{-\varphi/2}\|_{L^2(U)} \leq C_d \max(1,\rho^{-1/2}).  
\end{equation}
Let $\mb z \in \C^d$ with $|\mb z| \geq 1/4$. Apply Lemma \ref{lem:pointwise_bd_subharmonicity} to $f$ with $r = 1/8$. Then 
\begin{equation}\label{eq:pointwise_bd_f_geq18}
    |f(\mb z)| \leq C_d \| f\, e^{-\varphi/2}\|_{L^2(U)} \, \| e^{\varphi/2} \|_{L^{\infty}(\B_1(\mb z))}.
\end{equation}
For $\mb x \in \R^d$ with $|\mb x| \geq 1/4$ we have
\begin{align*}
    \sup_{|\mb w - \mb x| \leq 1} 2u(\mb w) &\leq 2u(\mb x) + 2C_{\Lip} + 2\rho, \\ 
    \sup_{|\mb w - \mb x| \leq 1} \rho u_0(\mb w) &\leq \rho u_0(\mb x) + \rho + \frac{1}{2}\rho, \\ 
    \sup_{|\mb x' + i \mb y' - \mb x| \leq 1} \frac{1}{2}\rho (\langle \mb y'\rangle - 1) &\leq \frac{1}{2}\rho, \\ 
    \sup_{|\mb w - \mb x| \leq 1} 20d\, \log |\mb w|_{\infty} &\leq 20d\, \log(1+|\mb x|) \leq -\frac{1}{2}\rho u_0(\mb x) + B.
\end{align*}
Here $B > 0$ is a constant and we may estimate
\begin{align*}
    B &= \sup_{r \geq 0} \left(20d\log (1+r) - \frac{1}{2}c_d \rho 1_{r \geq 10} \frac{r}{(\log (2+r))^2}\right) \\
    &\leq C_d \max(1, \log \rho^{-1}). 
\end{align*}
Combining these we obtain
\begin{equation*}
    |f(\mb x)| \leq C_d \, e^{C_{\Lip}}\, \max(\rho^{-C_d}, e^{2\rho})\, e^{u(\mb x)} e^{\frac{1}{2}c_d\rho \omega_0}\quad \text{for $\mb x \in \R^d$, $|\mb x| \geq 1/4$}.
\end{equation*}
Using the estimate \eqref{eq:bound_g_x_small} for $|\mb x| \leq 1/4$, the estimate \eqref{eq:prop_f_bdd_u} holds for all $\mb x \in \R^d$. Moreover, because of the term $e^{\frac{1}{2}c_d\rho \omega_0}$ in the upper bound, $f \in L^2(\R^d)$. 

Finally we show that $f$ has appropriate growth as $|\mb y| \to \infty$. We have
\begin{align*}
    20d \log |\mb x + i \mb y|_{\infty} + \rho u_0(\mb x + i \mb y) &\leq 20d \max(1,\log |\mb x|, \log |\mb y|) - \\
    &\qquad c_d \rho\, 1_{|\mb x|\geq 10} \frac{|\mb x|}{(\log (2+|\mb x|))^2} + \frac{\rho}{2} |\mb y| \\ 
    &\leq \frac{3}{2}\rho |\mb y| + A'
\end{align*}
for some constant $A' = A'(\rho, d) > 0$. So applying \eqref{eq:pointwise_bd_f_geq18} for $|\mb x+i\mb y| \geq 1/4$,
\begin{equation*}
    |f(\mb x + i \mb y)| \leq C_d \max(1, \rho^{-1/2}) \|e^{\varphi/2} \|_{L^{\infty}(\B_1(\mb x + i \mb y))} \leq C_d \max(1, \rho^{-1/2}) e^{A'/2+2\rho} e^{2\rho|\mb y|}.
\end{equation*}
Certainly $f$ is bounded for $|\mb z| \leq 1/4$ so \eqref{eq:psh_to_bdd_fourier_f_growth_y} holds for all $\mb z \in \C^d$.

\begin{remark*}
The quantitative bounds in Proposition \ref{prop:psh_to_bounded_fourier} can be improved in various ways, we don't try to optimize for this.
\end{remark*}

\section{Finishing the proof of the main theorem}\label{sec:finishing_pf_main_theorem}

\subsection{Proof of Proposition \ref{prop:existence_damping}}
We now construct weight functions adapted to line porous sets, and use Theorem \ref{thm:higher_dim_BM} to prove Proposition \ref{prop:existence_damping}. Let $\mb Y$ be $\nu$-porous on lines from scales $\mu$ to $h^{-1}$. Let $0 < \alpha < 1$ be the damping function parameter to be chosen later.

Consider the sequence of dyadic annuli $A_k = \{\mb x \in \R^d\, :\, 2^{k} \leq |\mb x| \leq 2^{k+1}\}$ for $k \geq 1$. Let 
\begin{equation}
    W_k = \frac{2^k}{k^s}
\end{equation}
where $s \in (0, 1)$ is a parameter to be chosen later (we will end up choosing $s = 0.2$). 
Let $\mc Q_k = \{\Q\}$ be a collection of finitely overlapping cubes of width $W_k$ so that $A_k \subset \bigcup_{\Q \in \mc Q_k} \frac{1}{2}\Q$, here $\frac{1}{2}\Q$ has the same center as $\Q$ and half the width. We require that
\begin{equation*}
    \bigcup_{\Q \in \mc Q_k} \Q \subset \{\mb x \in \R^d\, :\, 2^{k-1} \leq |\mb x| \leq 2^{k+2}\}.
\end{equation*}
For each $\Q \in \mc Q_k$, let $\eta_{\Q}$ be a bump function supported in $\Q$ and taking the value $1$ on $\frac{1}{2}\Q$. We construct $\eta_{\Q}$ by dilating a fixed bump function of width $1$, which gives the derivative estimate
\begin{align}
    \|D^a \eta_{\Q}\|_{\infty} &\lesssim_{d,a} W_k^{-a} \qquad \text{for all $a\geq 0$}. \label{eq:deriv_estimate_etaQ}
\end{align}
For all $ \mb x \in A_k$ we have 
\begin{equation*}
    \sum_{\Q \in \mc Q_k} \eta_{\Q}(\mb x) \in [1, C]
\end{equation*}
for some universal constant $C$. 
Let
\begin{align*}
    \mc S_{\mb Y,k} &= \{\Q \in \mc Q_k\, :\, \Q \cap (\mb Y\cap A_k) \neq \emptyset\}, \\ 
     \mb Y_k &= \bigcup_{\Q \in \mc S_{\mb Y,k}} \Q.
\end{align*}
Set 
\begin{align}\label{eq:defn_omega_j}
    \omega_k &= -\frac{2^k}{k^{\alpha}}\sum_{\Q \in \mc S_{\mb Y,k}} \eta_{\Q}.
\end{align}
Notice that $\supp \omega_k \subset \{\mb x \in \R^d\, :\, 2^{k-1}\leq |\mb x| \leq 2^{k+2}\}$, and that
\begin{equation}\label{eq:magnitude_omegak_in_A_k_Y}
    \omega_k(\mb x) \leq -\frac{2^k}{k^{\alpha}}\quad \text{for $\mb x \in \mb Y \cap A_k$.}
\end{equation}
The difference from Bourgain and Dyatlov's construction is that they take $\alpha = s$, and we allow for $\alpha$ to be much closer to $1$. Let $k_0 \geq 2$ be the smallest integer such that $W_{k_0} > \mu$ (this choice will be clear when we discuss the growth condition). Set 
\begin{equation}
    \omega = \sum_{k \geq k_0} \omega_k,
\end{equation}
notice that $\omega(\mb x) = 0$ for $|\mb x| \leq 2$. 
\begin{figure}
    \centering
    \includegraphics[width=0.5\linewidth]{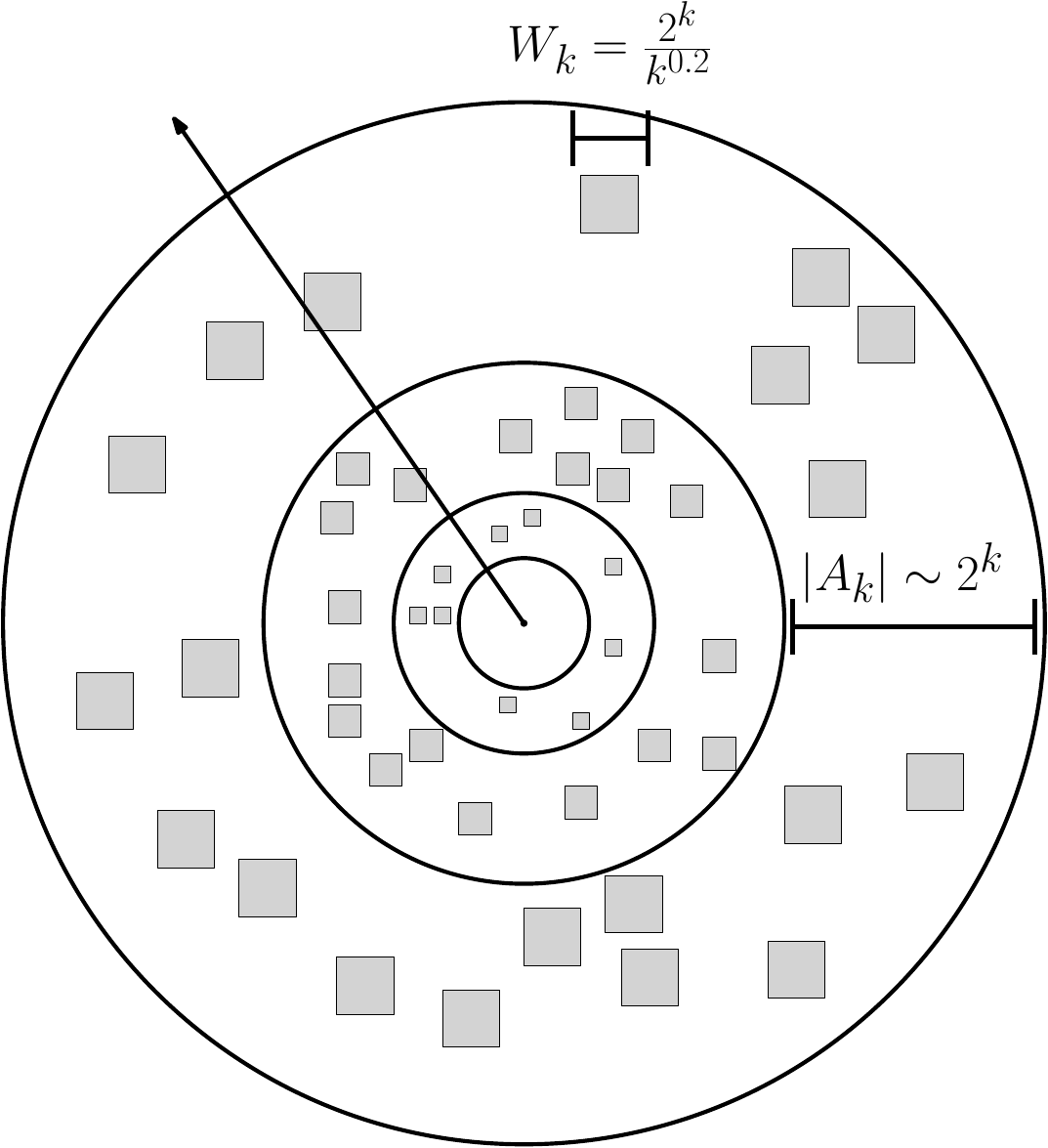}
    \caption{Within each dyadic annulus the weight is a sum of bump functions on boxes.}
    \label{fig:weight_func_constr}
\end{figure}
See Figure \ref{fig:weight_func_constr} for an image representing the weight. By \eqref{eq:magnitude_omegak_in_A_k_Y}, 
\begin{equation}\label{eq:lower_bd_omega_first}
    \omega(\mb x) \leq -\frac{1}{20}\frac{|\mb x|}{(\log (2+|\mb x|))^{\alpha}}\quad \text{for $|\mb x| > 2^{k_0}$ and $\mb x \in \mb Y$},
\end{equation}
so 
\begin{equation}\label{eq:lower_bd_omega}
    \omega(\mb x) \leq -\frac{1}{20}\frac{|\mb x|}{(\log (2+|\mb x|))^{\alpha}} + C(\mu) \quad \text{for all $\mb x \in \mb Y$}.
\end{equation}
Now we establish some regularity. For any $a\geq 0$, $k\geq 1$, we have 
\begin{equation}
    |D^a\omega_k| \lesssim_{a,d} W_k^{-a} 2^k k^{-\alpha} \sum_{\Q \in \mc S_{\mb Y,k}} 1_{\Q} \lesssim 
    2^{(1-a)k} k^{as-\alpha} 1_{\mb Y_k}
\end{equation}
where we use \eqref{eq:deriv_estimate_etaQ} for the first inequality and finite overlapping of the cubes in $\mc Q$ for the second inequality. 
As long as $3s < \alpha$, $\omega$ satisfies the regularity condition \eqref{eq:regularity_condition_0to3} with a constant $C_{\reg}$ that depends only on the dimension.

Next we discuss the growth condition. We have
\begin{equation}
    \mb Y_k \subset (\mb Y\cap A_k) + \B_{2W_k\sqrt{d}}.
\end{equation}
Because $\mb Y \subset [-3h^{-1},3h^{-1}]^d$, $\mb Y_k$ is empty if $2^{k}  > 3h^{-1}\sqrt{d}$ (this is the only place we use that $\mb Y \subset [-3h^{-1},3h^{-1}]^d$). Increasing $k_0$ if necessary by a value that only depends on $d$, we may assume $2W_k\sqrt{d} < h^{-1}$. 
If $k \geq k_0$ then $\mu < 2W_k\sqrt{d} < h^{-1}$ and by Lemma \ref{lem:basic_props_line_porous}\ref{lempart:nbhd_porous_set}, $\mb Y_k$ is $\nu/2$-porous on lines from scales $4W_k\sqrt{d}/\nu$ to $h^{-1}$ (this is a vacuous statement if $4W_k\sqrt{d}/\nu > h^{-1}$).

Let $\ell$ be any line. If $4W_k\sqrt{d}/\nu > 2^k/\sqrt{d}$ then $k^s < 4d/\nu$ and
\begin{align*}
    |\mb Y_k \cap \ell| \lesssim 2^k \lesssim_{\nu,d} 2^kk^{-s}. 
\end{align*}
Here $|\bullet|$ is the one-dimensional Lebesgue measure on $\ell$. 
Otherwise, we can split up $Y_k\cap \ell = \bigcup_{j} Y_k \cap \tau_j$ where each $\tau_j$ is a line segment on $\ell$ of length $2^k/\sqrt{d}$, and there are $\lesssim \sqrt{d}$-many line segments in the union.
Applying Corollary \ref{cor:line_porous_small_intersection} to each line segment and summing, 
\begin{equation}\label{eq:estimate_mass_Xj_l}
    |\mb Y_k \cap \ell| \lesssim_{\nu,d} 2^k\, k^{-s\gamma}\quad \text{ for all lines $\ell$}
\end{equation}
for some $\gamma = \gamma(\nu) > 0$. Thus if $\ell = \{t\mbhat y\, :\, t \in \R\}$ is a line through the origin, we see
\begin{align}\label{eq:int_over_lines_omega_k}
    2^{-k}\int_0^{\infty} |\omega_k(t\mbhat y)|\, dt \lesssim k^{-\alpha} |\mb Y_k \cap \ell| \lesssim 2^k\, k^{-(\alpha+s\gamma)}.
\end{align}
Let $G^*(r)$ be the growth function defined in \eqref{eq:defn_of_G*}. Let $r \in [2^k, 2^{k+1})$. We have the pointwise bound
\begin{align}
    G^*(r) &\lesssim \sup_{\mbhat y \in \mbS^{d-1}} 2^{-k} \int_{2^{k-1}}^{2^{k+2}} |\omega(t\mbhat y)|\, dt \nonumber\\ 
    &\lesssim \sup_{\mbhat y \in \mbS^{d-1}} 2^{-k}  \sum_{k-3 \leq j \leq k+3} \int_{0}^{\infty}|\omega_{j}(t\mbhat y)| \, dt \nonumber\\
    &\lesssim \frac{r}{(\log (2+r))^{\alpha+s\gamma}}. 
\end{align}
As long as $\alpha + s\gamma > 1$, the growth condition \eqref{eq:higher_dim_BM_growth_cond} is satisfied with a constant that depends on $\alpha+s\gamma$, $\nu$, and $d$.
We may choose $s = 0.2$ universally and $\alpha > 1 - 0.1\gamma(\nu)$. Then $-\alpha+3s < -0.3$ and $\alpha+s\gamma > 1+0.1\gamma$. 

The weight $\omega$ satisfies \eqref{eq:regularity_condition_0to3} and \eqref{eq:higher_dim_BM_growth_cond} with constants $C_{\reg}$ and $C_{\gr}$ that depend only on $\nu$ and $d$. We apply Theorem \ref{thm:higher_dim_BM} with spectral radius $\sigma/2 < 1$ to obtain a function $f \in L^2(\R^d)$ satisfying 
\begin{align*}
    \supp \hat f &\subset \B_{\sigma/2}, \\ 
    |f(\mb x)| &\geq \frac{1}{2} && \text{for $|\mb x| \leq c_d \sigma$}, \\ 
    |f(\mb x)| &\leq C(d,\mu)  \sigma^{-C_d} \exp\left(-c\, \sigma \frac{|\mb x|}{(\log (2+|\mb x|))^{\alpha}} \right) && \text{for $\mb x \in \mb Y$}, \\ 
    |f(\mb x) &\leq C(d,\mu) \sigma^{-C_d} && \text{for $\mb x \in \R^d$}.
\end{align*}
Here $c = c(\nu, d)$. 
Now let $\varphi: \R^d \to \R$ be a fixed Schwartz function with $\supp \hat \varphi \subset \B_1$ and $\int \varphi = 1$. Let $\varphi_{c_d\sigma/10}(\mb x) = \varphi((c_d\sigma/10)\mb x)$ so $\supp \hat \varphi \subset B_{c_d\sigma/10}$ and $\int \hat \varphi = 1$. Let $f_1 = f \varphi$. Then $\hat f_1 = \hat f * \hat \varphi$ and 
\begin{align*}
    \supp \hat f_1 &\subset \B_{\sigma}, \\ 
    |f_1(\mb x)| &\geq \frac{1}{2} && \text{for $|\mb x| \leq c_d \sigma/2$}, \\ 
    |f_1(\mb x)| &\leq C(d,\mu)  \sigma^{-C_d} \exp\left(-c\, \sigma \frac{|\mb x|}{(\log (2+|\mb x|))^{\alpha}} \right) && \text{for $\mb x \in \mb Y$}, \\ 
    |f_1(\mb x)| &\leq C(d,\mu) \sigma^{-C_d} \langle \mb x\rangle^{-d} && \text{for $\mb x \in \R^d$}.
\end{align*}
The last equation follows from $|\varphi(\mb x)| \leq C_d\langle \mb x\rangle^{-d}$.
Dividing through by $C(d,\mu) \sigma^{-C_d}$, we obtain a damping function with parameters $c_1 = \sigma$, $c_2 = c(d,\mu) \sigma^{C_d}$, $c_3 = c(\nu, d) \sigma$. 

\begin{remark*}
We may take 
\begin{equation}
    \alpha = 1 - 0.1\gamma(\nu) = 1 - c\frac{\nu}{|\log \nu|}
\end{equation}
for some absolute $c > 0$. 
\end{remark*}

\subsection{Proof of Theorem \ref{thm:FUP_higher_dim}}
Let 
\begin{itemize}
    \item $\mb X \subset [-1,1]^d$ be $\nu$-porous on balls from scales $h$ to 1,  
    \item $\mb Y \subset [-h^{-1}, h^{-1}]^d$ be $\nu$-porous on lines from scales $1$ to $h^{-1}$.
\end{itemize}
By Lemma \ref{lem:basic_props_line_porous}, for any $h < s < 1$ and $\eta \in [-h^{-1}s-5, h^{-1}s+5]^d$, the set 
\begin{equation*}
    s\mb Y + [-4,4]^d + \eta
\end{equation*}
is $\nu/2$-porous on lines from scale $10\sqrt{d}/\nu$ to $h^{-1}s$, and so by applying Proposition \ref{prop:existence_damping} with $\mu = 10\sqrt{d}/\nu$ it admits a damping function with parameters $c_1 = \nu/20\sqrt{d}$ and $c_2,c_3, \alpha \in (0,1)$ depending only on $\nu$ and $d$. Then by Theorem \ref{thm:FUP_conditional_damping}, there exists $\beta = \beta(\nu, d) > 0$ and $\widetilde C = \widetilde C(\nu, d) > 0$ so that for any $f \in L^2(\R^d)$ 
\begin{equation}
    \supp \hat f \subset \mb Y \, \Longrightarrow\, \| f1_{\mb X} \|_2 \leq \widetilde C h^{\beta}\, \| f\|_{2}.
\end{equation}

\appendix \section{Some miscellaneous pieces}\label{sec:loose_ends}
In this appendix we collect various technical statements used throughout the paper. Everything here is either standard or already in the literature. 
\subsection{Proof of Theorem \ref{thm:FUP_conditional_damping} from the version in \texorpdfstring{\cite{HanSchlag}}{Han-Schlag}}\label{sec:han_schlag_comparison}

We state Han and Schlag's theorem in their original terminology and prove that our version, Theorem \ref{thm:FUP_conditional_damping}, follows from their version. The differences are minor. 
First of all they have a slightly different definition of damping functions which is based on $\ell_1$ rather than $\ell_2$ norms. 
We denote the $\ell_1$ norm by $|\mb x|_1$ and the usual $\ell_2$ norm by $|\mb x|_2$.
\begin{definition}[\cite{HanSchlag}*{Definition 4.1}]\label{defn:damping_func_l1}
The set $\mb Y \subset \R^d$ \textit{admits an $\ell_1$ damping function} with parameters $c_1, c_2, c_3, \alpha\in (0, 1)$ if there exists a function $\psi \in L^2(\R^d)$ satisfying
\begin{align}
    \supp \hat  \psi &\subset [-c_1, c_1]^d, \\ 
    \|  \psi \|_{L^2([-1,1]^d)} &\geq c_2, \\ 
    | \psi(\mb x)| &\leq \langle \mb x \rangle^{-d} && \text{for all $\mb x \in \R^d$}, \\ 
    | \psi(\mb x)| &\leq \exp\left(-c_3 \frac{|\mb x|_1}{(\log (2+|\mb x|_1))^{\alpha}}\right)&& \text{for all $\mb x \in \mb Y$}.
\end{align}
\end{definition}
We note that in their paper they instead take $\supp \psi \subset [-c_1,c_1]^d$ and look at decay on the Fourier side but this is equivalent by taking a Fourier transform.
Because $|\mb x|_1 \leq \sqrt{d}\, |\mb x|_2$,
\begin{equation*}
    \frac{|\mb x|_1}{(\log (2+|\mb x|_1))^{\alpha}} \leq \frac{\sqrt{d}|\mb x|_2}{(\log (2+\sqrt{d}|\mb x|_2))^{\alpha}} \leq \sqrt{d}\frac{|\mb x|_2}{(\log (2+|\mb x|_2))^{\alpha}},
\end{equation*}
so an $\ell_2$ damping function with parameters $c_1,c_2,c_3,\alpha$ is an $\ell_1$ damping function with parameters $c_1,c_2,c_3/\sqrt{d},\alpha$.

Han and Schlag also use a slightly different definition of porosity, which we call box porosity.

\begin{definition}[\cite{HanSchlag}*{Definition 5.1}]\label{defn:box_porous}
Say that $\mb X \subset [-1,1]^d$ is \textit{box porous at scale $L \geq 3$ with depth $n$}, where $L$ is an integer, if the following holds. Denote by $\mc C_n$ the cubes obtained from $[-1,1]^d$ by partitioning it into congruent cubes of side length $L^{-n}$. The condition on $\mb X$ is that for all $\Q \in \mc C_n$ with $\Q \cap \mb X \neq \emptyset$, there exists $\Q' \in \mc C_{n+1}$ so that $\Q' \subset \Q$ and $\Q' \cap \mb X = \emptyset$. 
\end{definition}

Now we show that a $\nu$-porous set is also box porous. 
\begin{lemma}\label{lem:compare_porosity}
Let $\mb X \subset [-1,1]^d$ be $\nu$-porous on balls from scale $h$ to $1$. Then $\mb X$ is box porous at scale $L = \lceil \nu^{-1}\sqrt{d}\rceil$ with depth $n$ for all $n\geq 0$ with $L^{-n} \geq h$. 
\end{lemma}
\begin{proof}
Let $L$, $n$ be as above and let $\mc C_n$, $\mc C_{n+1}$ be as in the definition of box porosity. Let $\Q \in \mc C_n$.  Let $\B \subset \Q$ be a ball of diameter $L^{-n}$. By the definition of porosity, there is some $\mb x \in \B$ such that $\B_{\nu L^{-n}}(\mb x) \cap \mb X = \emptyset$. Let $\Q' \in \mc C_{n+1}$ be a cube containing $\mb x$ and $\Q' \subset \Q$. Then because $\nu L^{-n} \geq L^{-n-1}\sqrt{d}$, we have $\Q' \subset \B_{\nu L^{-n}}(\mb x)$ and $\Q' \cap \mb X = \emptyset$ as needed.
\end{proof}

We can now state Han and Schlag's theorem exactly as it appears in \cite{HanSchlag}. 
\begin{theorem}[\cite{HanSchlag}*{Theorem 5.1}]\label{thm:han_schlag_original}
Suppose that
\begin{itemize}
    \item $\mb X \subset [-1,1]^d$ is box porous at scale $L \geq 3$ with depth $n$, for all $n \geq 0$ with $L^{n+1} \leq N$.
    \item $\mb Y \subset [-N, N]^d$ is such that for all $n \geq 0$ with $L^{n+1} \leq N$ one has that for all $\eta \in [-NL^{-n} - 3, NL^{-n}+3]^d$
    the set 
    \begin{equation*}
        L^{-n} \mb Y + [-4,4]^d + \eta
    \end{equation*}
    admits an $\ell_1$ damping function with parameters $c_1 = (2L)^{-1} \in (0, \frac{1}{2}]$ and $c_2, c_3 \in (0, 1)$.
\end{itemize}
Assume $0 < c_3 < c_3^*(d)$. Then there exists $\beta = \beta(L, c_2, c_3, d, \alpha) > 0$ and $\widetilde C = \widetilde C(L, c_2, c_3, d, \alpha) > 0$ so that any $f \in L^2(\R^d)$ with $\supp \hat f \subset \mb Y$ satisfies 
\begin{equation*}
    \| f 1_{\mb X}\|_2 \leq \widetilde C N^{-\beta} \| f \|_2
\end{equation*}
for all $N \geq N_0(L, c_2, c_3, d, \alpha)$. 
\end{theorem}

We prove that Theorem \ref{thm:han_schlag_original} implies Theorem \ref{thm:FUP_conditional_damping}. 

\begin{proof}[Proof of Theorem \ref{thm:FUP_conditional_damping} from Theorem \ref{thm:han_schlag_original}]
Suppose that the hypotheses of Theorem \ref{thm:FUP_conditional_damping} are satisfied with parameters $\nu$, $h$ and $c_1 = \nu/(20\sqrt{d})$, $c_2,c_3,\alpha \in (0,1)$. 
Let $L = \lceil \nu^{-1}\sqrt{d}\rceil$ and let $N = \lceil h^{-1} \rceil$. Notice that $c_1 < \frac{1}{2L}$ as needed.

For all $h < s < 1$ and $\eta \in [-h^{-1}s - 5, h^{-1}s+5]^d$ the set 
\begin{equation*}
    s\mb Y + [-4,4]^d + \eta
\end{equation*}
admits an $\ell_1$ damping functions with parameters $c_1 = \nu/(20\sqrt{d})$ and $c_2,c_3/\sqrt{d},\alpha$. Because it is a strictly stronger property for $\mb Y$ to admit a damping function with a larger $c_3$ parameter, we can assume for free that $c_3 < c_3^*(d)$. 

We have $\mb Y \subset [-N,N]^d$. Let $n\geq 0$ be such that $L^{n+1} \leq N$. By Lemma \ref{lem:compare_porosity}, $\mb X$ is box porous at scale $L$ with depth $n$. Also, $h < L^{-n} \leq 1$. We have
\begin{equation}
    [-NL^{-n} - 3, NL^{-n}+3]^d \subset [-h^{-1}L^{-n} - 5, h^{-1} L^{-n} + 5]^d, 
\end{equation}
so for any $\eta \in [-NL^{-n} - 3, NL^{-n}+3]^d$, the set $L^{-n} \mb Y + [-4,4]^d + \eta$ admits an $\ell_1$ damping function with parameters $c_1 = (2L)^{-1}$, $c_2 \in (0, 1)$, and $0 < c_3 < c_3^*(d)$. By Theorem \ref{thm:han_schlag_original} there exists 
\begin{align*}
    \beta &= \beta(L, c_2, c_3, d, \alpha) = \beta(\nu, c_2, c_3, d, \alpha) > 0, \\ 
    \widetilde C &= \widetilde C(L, c_2, c_3, d, \alpha) = \widetilde C(\nu, c_2, c_3, d, \alpha) > 0
\end{align*} 
such that any $f \in L^2(\R^d)$ with $\supp \hat f \subset \mb Y$ satisfies 
\begin{equation*}
    \| f1_{\mb X}\|_2 \leq \widetilde C h^{\beta} \|f \|_2
\end{equation*}
for all $h < 1/100$. We absorbed the condition that $N > N_0$ in Theorem \ref{thm:han_schlag_original} into the constant $\widetilde C$. 
\end{proof}

\subsection{Basic properties of line porous sets}

\begin{lemma}\label{lem:basic_props_line_porous}
Let $\mb X \subset \R^d$ be $\nu$-porous on lines from scales $\alpha_0$ to $\alpha_1$.
\begin{enumerate}[leftmargin=*,labelindent=
\parindent,=itemsep=0pt,label=(\alph*), ref=(\alph*)]
    \item Let $\alpha_0 < r < \alpha_1$ and let $\nu' < \nu$.
    Then $\mb X + \B_r$ is $\nu'$-porous on lines from scales $r/(\nu - \nu')$ to $\alpha_1$.\label{lempart:nbhd_porous_set}

    \item For any $s > 0$, the dilate $s \cdot \mb X$ is $\nu$-porous on lines from scales $s\, \alpha_0$ to $s\, \alpha_1$. 
    
    \item Let $\ell \subset \R^d$ be a line. Let $\mb X|_{\ell} = \mb X \cap \ell$, and view $\mb X|_{\ell}$ as a subset of $\R$. Then $\mb X|_{\ell}$ is $\nu$-porous from scales $\alpha_0$ to $\alpha_1$. \label{lemitem:line_porous_restriction}
\end{enumerate}
\end{lemma}
\begin{proof}~
\begin{enumerate}[label=(\alph*)]
    \item Let $\tau$ be a line segment of length $R$ with $r/(\nu - \nu') < R < \alpha_1$. Let $\mb x \in \tau$ be such that $\B_{\nu R}(\mb x) \cap \mb X = \emptyset$. Then $\B_{\nu' R}(\mb x) \cap (\mb X + \B_{(\nu - \nu')R}) = \emptyset$ as well. By the choice of $R$, $(\nu - \nu')R > r$ as needed.
    
    \item Let $\tau$ be a segment of length $R$ with $s\, \alpha_0 < R < s\,\alpha_1$. There is some $\mb x \in s^{-1}\cdot \tau$ such that $\B_{s^{-1}\nu R}(\mb x) \cap \mb X = \emptyset$. Then $\B_{\nu R}(s\mb x) \cap (s \cdot \mb X) = \emptyset$. 

    \item Let $\tau \subset \ell$ be a segment of length $R$. There is some $\mb x \in \tau$ such that $\B_{\nu R}(\mb x) \cap \mb X = \emptyset$. Then $(\B_{\nu R}(\mb x) \cap \ell) \cap \mb X|_{\ell} = \emptyset$.
\end{enumerate}
\end{proof}

\begin{lemma}\label{lem:relationship_box_porosity_regularity}
Let $\mb X \subset [-1,1]^d$ be box porous at scale $L$ with depth $n$ for all $0\leq n < N$. Then with $|\mb X|$ the Lebesgue measure,
\begin{equation}
    |\mb X| \leq 2^d\, (1-L^{-d})^N.
\end{equation}
\end{lemma}
\begin{proof}
We proceed by induction. Define
\begin{equation*}
    J(n) = \sup_{\Q \in \mc C_n} |\mb X \cap \Q|. 
\end{equation*}
where $\mc C_n$ is the family of congruent $L^{-n}$ cubes described in Definition \ref{defn:box_porous}. If $\mb X$ is box porous at scale $n$, then 
\begin{align*}
    J(n) &\leq (L^d - 1)J(n+1),\\
    |\mb X| &\leq 2^d J(0) \leq 2^d (L^d-1)^N J(N). 
\end{align*}
Using the trivial bound $J(N) \leq L^{-Nd}$, 
\begin{equation*}
    |\mb X| \leq 2^d (1 - L^{-d})^N.
\end{equation*}
\end{proof}

\begin{lemma}\label{lem:relationship_porosity_regularity}
Let $\mb X \subset \R^d$ be $\nu$-porous on balls from scales $\alpha_0$ to $\alpha_1$. Then there is some $C, \gamma > 0$ depending only on $\nu$ and $d$ such that for any ball $\B$ of radius $\alpha_0 < R < \alpha_1$, 
\begin{equation*}
    |\mb X \cap B| \leq C R^d \left(\frac{\alpha_0}{R}\right)^{\gamma}.
\end{equation*}
\end{lemma}
\begin{proof}
We use Lemma \ref{lem:compare_porosity} to reduce to proving the statement for box porous sets. Let $\Q$ be a $2R$-cube containing $\B$. Let $\mb X' = R^{-1}\cdot (\Q\cap \mb X) - \mb v \subset [-1,1]^d$ be a translated and rescaled copy of $\Q\cap \mb X$. Then $\mb X'$ is $\nu$-porous on balls from scales $\frac{\alpha_0}{R}$ to $1$, so it is also box porous at scale $L = \lceil \nu^{-1}\sqrt{d}\rceil$ with depth $n$ for all $n\geq 0$ with $L^{-n} \geq \frac{\alpha_0}{R}$. Let $N > 0$ be the smallest integer so that $L^{-N} < \frac{\alpha_0}{R}$. By Lemma \ref{lem:relationship_box_porosity_regularity},
\begin{equation*}
    |\mb X \cap \B| \leq R^d |\mb X'| \leq 2^d R^d\, (1-L^{-d})^{N}.
\end{equation*}
Let $\gamma = \gamma(\nu, d) > 0$ be such that $L^{-\gamma} = 1 - L^{-d}$. Then 
\begin{equation*}
    |\mb X \cap \B| \leq 2^dR^d L^{-N\gamma} \leq 2^d\, R^d\, \left(\frac{\alpha_0}{R}\right)^{\gamma}.
\end{equation*}
\end{proof}

\begin{remark*}
We can take 
\begin{equation}
    \gamma \geq c\frac{L^{-d}}{\log L} = c_d \frac{\nu^d}{|\log \nu|}.
\end{equation}
\end{remark*}

By combining Lemma \ref{lem:relationship_porosity_regularity} and Lemma \ref{lem:basic_props_line_porous}\ref{lemitem:line_porous_restriction}, we find that line porous sets have small intersections with lines. 
\begin{corollary}\label{cor:line_porous_small_intersection}
Let $\mb Y \subset \R^d$ be $\nu$-porous on lines from scales $\alpha_0$ to $\alpha_1$. Let $\tau$ be a line segment of length $\alpha_0 < R < \alpha_1$. Then there is some $C, \gamma > 0$ depending only on $\nu$ such that
\begin{align*}
    |\tau \cap \mb Y| \leq C R\, \left(\frac{\alpha_0}{R}\right)^{\gamma}.
\end{align*}
Here $|\bullet|$ is the one-dimensional Lebesgue measure on $\tau$. 
\end{corollary}
\begin{proof}
Let $\tau$ lie on the line $\ell$. By Lemma \ref{lem:basic_props_line_porous}\ref{lemitem:line_porous_restriction}, $\mb Y|_{\ell}$ is $\nu$-porous. By Lemma \ref{lem:relationship_porosity_regularity} in $d = 1$ we obtain the result. 
\end{proof}
\begin{remark*}
We can take $\gamma \geq c \frac{\nu}{|\log \nu|}$.
\end{remark*}

\subsection{The Paley--Wiener criterion}\label{subsec:PaleyWienerProof}
We sketch a proof of Theorem \ref{thm:paley_wiener}, the Paley--Wiener criterion for functions with bounded Fourier support. See \cite{HormanderVol1}*{Theorem 7.3.1} for a full proof. 

Suppose $f \in L^2(\R^d)$ and $\supp \hat f \subset \B_{\sigma/2\pi}$. Then we have 
\begin{equation*}
    f(\mb x) = \int_{|\xi| \leq \sigma/2\pi} \hat f(\xi) e^{2\pi i\, \mb x\cdot \xi}\, d\xi. 
\end{equation*}
We can define
\begin{equation*}
    \tilde f(\mb z) = \int_{|\xi| \leq \sigma/2\pi} \hat f(\xi) e^{2\pi i\, \mb z\cdot \xi}\, d\xi. 
\end{equation*}
for any $\mb z \in \C^d$, and we find
\begin{equation*}
    \tilde f(\mb z) \leq C_{\sigma,d}\| f\|_2 \sup_{|\xi| \leq \sigma/2\pi} e^{2\pi i\, \mb z \cdot \xi} = C_{\sigma,d} \| f \|_2\, e^{\sigma |\Im \mb z|}
\end{equation*}
as desired. 

Now suppose $\tilde f$ is analytic on $\C^d$, $f = \tilde f|_{\R^d}$ is a Schwartz function, and
\begin{equation*}
    \tilde f(\mb x + i \mb y) \leq C e^{\sigma |\mb y|}.
\end{equation*}
We have
\begin{equation*}
    \hat f(\xi) = \int_{\R^d} f(\mb x) e^{-2\pi i\, \mb x\cdot \xi}\, d\mb x.
\end{equation*}
Fix $\xi$ and suppose $|\xi| > \sigma/2\pi$.
Let 
\begin{align*}
    G_{\mb x_0, \mbhat \xi}(z) &= \tilde f(\mb x_0 + z\mbhat \xi) e^{-2\pi i\, (\mb x_0 + z\hat \xi) \cdot \xi}  
\end{align*}
Notice $G_{\mb x_0, \mbhat \xi}(z)$ is an analytic function of $z$, and because $|\tilde f(\mb x_0 + a\hat \xi + ib \hat \xi)| \leq Ce^{\sigma |b|}$, we have 
\begin{equation*}
    G_{\mb x_0, \mbhat \xi}(a-ib) \leq C e^{-(2\pi |\xi|-\sigma)b}
\end{equation*}
and so $|G(z)|$ exponentially decays as $\Im z \to -\infty$. By contour integration,
\begin{equation*}
     \int_{-\infty}^{\infty} G_{\mb x_0, \mbhat \xi}(t)\, dt = \lim_{b \to \infty} \int_{-\infty}^{\infty} G_{\mb x_0, \mbhat \xi}(t - ib)\, dt = 0.
\end{equation*}
We write a general $\mb x \in \R^d$ as $\mb x_0 + t\mbhat \xi$ where $\mb x_0 \in \mb \xi^{\perp}$, and integrate over $\xi^{\perp}$ to conclude that $\hat f(\mb \xi) = 0$. 

\subsection{Hilbert transform of the derivative from integrals over lines}\label{subsec:hilbert_from_second_deriv}

The following Lemma is due to Semyon Dyatlov. 
We prove it directly, but Dyatlov originally discovered it indirectly by proving via distributions that Proposition \ref{prop:Ew_is_psh}\ref{item_Ew_psh:second_deriv} suffices for $E\omega + C|\mb y|$ to be plurisubharmonic.

\begin{lemma}
Let $d \geq 2$, $\omega \in C_0^2(\R^d)$. 
Suppose that for all $\ell = \{\mb x + t \mbhat y\} \subset \R^d$ a line and $\mbhat v \perp \mbhat y$ we have
\begin{equation*}
    \left|\int_{-\infty}^{\infty} \langle (D^2\omega(\mb x + t \mbhat y))\mbhat v, \mbhat v\rangle\, \frac{dt}{\pi}\right| \leq C. 
\end{equation*}
Then for any such line $\ell$, we have
\begin{equation*}
    \| H[\omega|_{\ell}'] \|_{\infty} \leq C. 
\end{equation*}
\end{lemma}
\begin{proof}
It suffices to prove this when $d = 2$ because we can always restrict to a plane containing $\ell$ when $d > 2$. 
Let $P(x,y) = \frac{1}{\pi}\frac{|y|}{x^2+y^2}$ be the Poisson kernel. We have
\begin{align*}
    \Delta P (x,y) = 2\partial_y P(x, 0+)\delta_{\ell_1}
\end{align*}
where $\ell_1 = \{(t, 0)\, :\, t \in \R\}$ is the $x$-axis and the derivative is in the upward pointing normal direction to $\ell_1$. The right hand side is the distribution $\partial_y P(x, 0+)$ on $\R$ pushed forward to the line $\ell_1$. By Lemma \ref{lem:normal_derivative_harmonic_extension}, 
\begin{align*}
    \partial_y P(x, 0+) = -2\frac{d}{dx}H[\delta_0]
\end{align*}
viewed as a distribution on $\R$. Taking adjoints,
\begin{equation*}
    \langle  P, \Delta \omega\rangle = \langle \Delta P, \omega\rangle = -2H[\omega|_{\ell_1}'](0).
\end{equation*}
We have
\begin{equation*}
    \int_{-\infty}^{\infty} \langle (D^2\omega(\mb x + t \mbhat y))\mbhat v, \mbhat v\rangle\, dt = \int_{-\infty}^{\infty} \Delta \omega(\mb x + t \mbhat y)\, dt.
\end{equation*}
In radial coordinates the Poisson kernel is given by $P(r, \theta) = \frac{|\sin \theta|}{\pi |r|}$, so
\begin{align*}
 \int P(x,y)\, \Delta \omega(x,y)\, dxdy &= \int_0^{\pi} \int_{-\infty}^{\infty} P(r,\theta)\, \Delta \omega(r, \theta)\, |r| drd\theta  \\ 
 &= \frac{1}{\pi} \int_0^{\pi} \int_{-\infty}^{\infty} |\sin \theta|\, \Delta \omega(r, \theta)\,  drd\theta  \leq 2C.
\end{align*} 
\end{proof}
A subtle distinction is that we required upper and lower bounds in Proposition \ref{prop:Ew_is_psh}\ref{item_Ew_psh:hilbert} but only lower bounds in \ref{item_Ew_psh:second_deriv}, so \ref{item_Ew_psh:hilbert} does not technically follow from \ref{item_Ew_psh:second_deriv} as stated. We could just as well require upper and lower bounds in \ref{item_Ew_psh:second_deriv}, and then \ref{item_Ew_psh:hilbert} would follow. 

\bibliography{references}

\end{document}